\documentclass[12pt]{amsart}
\usepackage{amssymb}
\usepackage{csquotes}
\usepackage[pagebackref=false,colorlinks,linkcolor=blue,citecolor=blue]{hyperref}
\usepackage{mathrsfs,graphicx,amsthm, mathtools}
\usepackage[all]{xy}
\usepackage{tikz}
\usepackage{amsmath}
\input xy
\xyoption{all}

\newtheorem{lemma}{Lemma}[section]
\newtheorem{corollary}[lemma]{Corollary}
\newtheorem{theorem}[lemma]{Theorem}
\newtheorem{proposition}[lemma]{Proposition}
\newtheorem{condition}[lemma]{Condition}

\theoremstyle{definition}
\newtheorem{remark}[lemma]{Remark}
\newtheorem{definition}[lemma]{Definition}

\newtheorem{example}[lemma]{Example}

\newtheorem{question}[lemma]{Question}
\DeclareMathOperator{\Mod}{Mod}
\DeclareMathOperator{\modd}{mod}
\DeclareMathOperator{\End}{End}

\DeclareMathOperator{\add}{add}
\DeclareMathOperator{\Hom}{Hom}

\DeclareMathOperator{\Ext}{Ext}

\DeclareMathOperator{\ind}{ind}
\DeclareMathOperator{\rad}{rad}

\DeclareMathOperator{\mcm}{\mathcal{M}}
\DeclareMathOperator{\mcc}{\mathcal{C}}

\DeclareMathOperator{\dep}{dp}
\DeclareMathOperator{\im}{Im}
\DeclareMathOperator{\Coker}{Coker}

\newtheorem*{theorem a*}{Theorem A}
\newtheorem*{theorem b*}{Condition B}
\newtheorem*{theorem c*}{Theorem C}
\newtheorem*{theorem d*}{Theorem D}
\newtheorem*{theorem e*}{Theorem E}
\newtheorem*{theorem f*}{Theorem F}

\definecolor{blac}{RGB}{0,0,1}
\begin{document}

\title{The radical of functorially finite subcategories}

\author{Raziyeh Diyanatnezhad}
\address{Department of Pure Mathematics\\
Faculty of Mathematics and Statistics\\
University of Isfahan\\
P.O. Box: 81746-73441, Isfahan, Iran}
\email{r.diyanat@sci.ui.ac.ir}

\author{Alireza Nasr-Isfahani}
\address{Department of Pure Mathematics\\
Faculty of Mathematics and Statistics\\
University of Isfahan\\
P.O. Box: 81746-73441, Isfahan, Iran\\ and School of Mathematics, Institute for Research in Fundamental Sciences (IPM), P.O. Box: 19395-5746, Tehran, Iran}
\email{nasr$_{-}$a@sci.ui.ac.ir / nasr@ipm.ir}



\begin{abstract}
Let $\Lambda$ be an artin algebra and $\mcc$ be a functorially finite subcategory of $\modd\Lambda$ which contains $\Lambda$ or $D\Lambda$. We use the concept of the infinite radical of $\mcc$ and show that $\mcc$ has an additive generator if and only if $\rad^\infty_{\mcc}$ vanishes. In this case we describe the morphisms in powers of the radical of $\mcc$ in terms of its irreducible morphisms. Moreover, under a mild assumption, we prove that $\mcc$ is of finite representation type if and only if any family of monomorphisms (epimorphisms) between indecomposable objects in $\mathcal{C}$ is noetherian (conoetherian). Also, by using injective envelopes, projective covers, left $\mcc$-approximations and right $\mcc$-approximations of simple $\Lambda$-modules, we give other criteria  to describe whether $\mcc$ is of finite representation type. In addition, we give a nilpotency index of the radical of $\mcc$ which is independent from the maximal length of indecomposable $\Lambda$-modules in $\mcc$.
\end{abstract}

\maketitle

\section{Introduction}
Let $\Lambda$ be an artin algebra and $\modd\Lambda$ be the category of finitely generated left $\Lambda$-modules. We denote the {\em (Jacobson) radical} of $\modd\Lambda$ by $\rad_\Lambda$. We remind that $\rad_\Lambda(X,Y)$ is the set of all non-isomorphisms from $X$ to $Y$, for two indecomposable $\Lambda$-modules $X$ and $Y$. The powers of $\rad_\Lambda(X,Y)$ are defined in the natural way. We denote by $\rad_\Lambda^\infty(X,Y)$ the intersection of all powers $\rad^i_\Lambda(X,Y)$ of $\rad_\Lambda(X,Y)$ with $i\geq1$.
A morphism $g:X\longrightarrow Y$ in $\modd\Lambda$ is called {\em irreducible} if it is neither a section nor a retraction and for every factorization $g=hf$, either $h$ is a retraction or $f$ is a section. There is a close relationship between irreducible morphisms in $\modd\Lambda$ and the radical of $\modd\Lambda$. Bautista proved that a morphism $f:X\longrightarrow Y$ between two indecomposable $\Lambda$-modules $X$ and $Y$ is irreducible if and only if $f\in\rad_\Lambda(X,Y)\backslash\rad_\Lambda^2(X,Y)$ \cite{B}.
Also a description of morphisms in $\rad_\Lambda^d(X,Y)$ with two indecomposable $\Lambda$-modules $X$ and $Y$ and integer $d\geq2$ was given in \cite[ V. Proposition 7.4]{ARS}.
The radical of the module category is one of the fundamental tools for studying representation theory of artin algebras.
For example, by using the radical of the module category of artin algebras, a classification of artin algebras based on the representation type was given. Indeed, it is well known that an artin algebra $\Lambda$ is of finite representation type if and only if the radical of $\modd\Lambda$ vanishes \cite{KS,Sk}.
More precisely, by the Auslander's works, if $\Lambda$ is of finite representation type then $\rad_\Lambda^\infty=0$ \cite{AuL,Au2}. The other side of this result is a consequence of the Harada--Sai lemma.
Coelho et al. deepened the Auslander's result by showing that $(\rad_\Lambda^\infty)^2=0$ if and only if $\Lambda$ is of finite representation type \cite{Co}. Then Chaio and Liu strengthened these two results. They showed that $\Lambda$ is of finite representation type if and only if the projective covers (resp. injective envelopes) of the
simple $\Lambda$-modules do not lie in the infinite radical or the compositions of projective covers and injective envelopes of the
simple $\Lambda$-modules do not lie in the square of the infinite radical \cite{Ch-Li}. By the above facts, if $\Lambda$ is of finite representation type, then there exists an integer $d\geq1$ such that $\rad_\Lambda^d=0$. This integer is called the {\em nilpotency index} of the radical of $\modd\Lambda$. By Harada--Sai lemma, $2^b-1$ is the nilpotency index of $\rad_\Lambda$, where $b$ is the maximal length of indecomposable $\Lambda$-modules.
Chaio and Liu in \cite{Ch-Li} defined the {\em depth} of a morphism in $\modd\Lambda$ and showed how one can compute the nilpotency index of $\rad_\Lambda$ in terms of the depth of the compositions of projective covers and injective envelopes of simple $\Lambda$-modules.

The notion of functorially finite subcategories of $\modd\Lambda$ was introduced by Auslander and Smal{\o} in \cite{AS,AS1}. In fact, these subcategories appeared in connection with studying the problem of which subcategories of $\modd\Lambda$ have almost split sequences. Functorially finite subcategories play an important role in the representation theory of artin algebras. The contravariantly finite resolving subcategories, the subcategory of finitely generated Gorenstein projective modules and $n$-cluster tilting subcategories are some certain kind of functorially finite subcategories which have been already studied. For example, Auslander and Reiten showed that the notion of a contravariantly finite resolving subcategory has a close relationship with the tilting theory. They classified contravariantly finite resolving subcategories over an artin algebra of finite global dimension in terms of
cotilting modules \cite{AR0}. Let $n$ be a positive integer. $n$-cluster tilting subcategories are one of the most important concepts of the higher Auslander--Reiten theory that was introduced by Iyama in 2004 \cite{I2, I1} and was further studied and developed in several papers \cite{I3, I4}. These subcategories were introduced during looking for a higher dimensional version of the Auslander's correspondence.

We recall that the subcategory $\mcc$ of $\modd\Lambda$ is called of finite representation type if the number of isomorphism classes of indecomposable objects in $\mcc$ is finite. The representation type of functorially finite subcategories of $\modd\Lambda$ is an interesting topic for studding. For instance, the finiteness of representation type of an $n$-cluster tilting subcategory of $\modd\Lambda$ for $n\geq 2$ is one of the main questions in the higher Auslander--Reiten theory that was posed by Iyama \cite{I3}.
This question is still unanswered, but so far many attempts have been made to answer it \cite{DN, EN, EN0, FN}.
Also, the representation type of the subcategory of finitely generated Gorenstein projective modules of $\modd\Lambda$ has been already studied. We remind that $\Lambda$ is called of finite $\mathrm{CM}$-type if the number of isomorphism classes of the indecomposable
finitely generated Gorenstein projective $\Lambda$-modules is finite. There are some results by Chen, Beligiannis and others showing when a Gorenstein artin algebra is of finite $\mathrm{CM}$-type \cite{C, Be, FN}.

In this paper we try to find criteria that can use in evaluating the representation type of a functorially finite subcategory of $\modd\Lambda$. In fact, we give equivalent conditions to the finiteness of the representation type of the functorially finite subcategory $\mcc$ of $\modd\Lambda$, by using the radical of $\mcc$. More precisely, we first use the concepts of the noetherian and the conoetherian family of morphisms in $\mcc$ and show that if any family of morphisms between indecomposable objects in a functorially finite subcategory $\mathcal{C}$ of $\modd\Lambda$ is noetherian (conoetherian), then $\mathcal{C}$ is of finite representation type (see Proposition \ref{pro6}). Then, we prove the following theorem.

\begin{theorem a*}$($Theorem \ref{main}$)$\label{A}
Let $\mcc$ be a functorially finite subcategory of $\modd\Lambda$.
\begin{itemize}
\item[(a)]
 If $\mcc$ contains $D\Lambda$, then the following are equivalent.
\begin{itemize}
\item[(i)]
There exists $t\in\mathbb{N}$ such that for any $X\in\ind\mcc$, $\rad^t_{\mcc}(X,-)=0$.
\item[(ii)]
For any $X\in\ind\mcc$, there exists $t\in\mathbb{N}$ such that $\rad^t_{\mcc}(X,-)=0$.\item[(iii)]$\mcc$ is of finite representation type.
\item[(iv)]
$\rad^\infty_{\mcc}=0$.
\end{itemize}
\item[(b)]
If $\mcc$ contains $\Lambda$, then the following are equivalent.
\begin{itemize}
\item[(i)]
There exists $t\in\mathbb{N}$ such that for any $X\in\ind\mcc$, $\rad^t_{\mcc}(-,X)=0$.
\item[(ii)]
For any $X\in\ind\mcc$, there exists $t\in\mathbb{N}$ such that $\rad^t_{\mcc}(-,X)=0$.
\item[(iii)]
$\mcc$ is of finite representation type.
\item[$(\mathrm{iv})$]$\rad^\infty_{\mcc}=0$.
\end{itemize}
\end{itemize}
\end{theorem a*}
By the Auslander's works, we know that $\modd\Lambda$ is of finite representation type if and only if every family of monomorphisms (resp. epimorphisms) between indecomposable objects in $\modd\Lambda$ is noetherian (resp. conoetherian) (for example see \cite[Theorem 3.1]{Au2} and \cite[Theorem A]{AuL}). It is natural to ask when these results are satisfied for a functorially finite subcategory $\mathcal{C}$ of $\modd\Lambda$. We define the {\em strongly minimal element} concept (see Definition \ref{sm}) and show that these results are satisfied for a functorially finite subcategory $\mathcal{C}$ of $\modd\Lambda$ which contains $\Lambda$ or $D\Lambda$ provided that the following condition is satisfied.
\begin{theorem b*}$($Condition \ref{c1}$)$\label{B}
For any $\mcc$-module $F$, $X\in\mcc$ and $0\neq x\in F(X)$, there is a morphism $f:X\longrightarrow X^\prime$ in $\mcc$ such that $F(f)(x)$ is a strongly minimal element in $F(X^\prime)$.
\end{theorem b*}
Indeed, we prove the following theorem.
\begin{theorem c*}$($Theorem \ref{th11}$)$\label{C}
Let $\mcc$ be a functorially finite subcategory of $\modd\Lambda$ which contains $\Lambda$ or $D\Lambda$. Then $\mcc$ is of finite representation type if and only if the following conditions hold.
\begin{itemize}
\item[(i)]
Condition B is satisfied.
\item[(ii)]
Any family of monomorphisms between indecomposable objects in $\mathcal{C}$ is noetherian.
\end{itemize}
\end{theorem c*}
Note that when the category $\mcc$ in the above theorem is closed under the image of morphisms, Condition B is automatically satisfied. Therefore, the theorem reduce to this result that any family of monomorphisms (resp. epimorphisms) between indecomposable objects in $\mathcal{C}$ is noetherian (resp. conoetherian) if and only if $\mathcal{C}$ is of finite representation type (see Corollary \ref{cor2}).

When we assume that a functorially finite subcategory of $\modd\Lambda$ which contains $\Lambda$ or $D\Lambda$ is of finite representation type, we can describe its morphisms in terms of its irreducible morphisms.

\begin{theorem d*}$($Theorem \ref{sumirr}$)$\label{D}
Let $\mcc$ be a functorially finite subcategory of $\modd\Lambda$ which contains $\Lambda$ or $D\Lambda$ and $f\in\rad_{\mcc}( A,B)$ be a nonzero morphism with $A,B\in\ind\mcc$. If $\mcc$ is of finite representation type, then $f$ is a sum of compositions of irreducible morphisms in $\mcc$ between indecomposable $\Lambda$-modules.
\end{theorem d*}
In the following theorem, we use some special morphisms to give another criterion for recognizing the finiteness of the representation type of a functorially finite subcategory $\mcc$ of $\modd\Lambda$. Let $\iota_S$ be the injective envelope of $S$, $\pi_S$ be the projective cover $S$, $l_S$ be the left $\mcc$-approximation of $S$ and $r_S$ be the right $\mcc$-approximation of $S$, for each simple $\Lambda$-module $S$. We refer the readers to Definition \ref{def4.1} for the definition of the depth of a morphism.
\begin{theorem e*}$($Theorem \ref{dp1}$)$\label{E}
Let $\mcc$ be a functorially finite subcategory of $\modd\Lambda$.
\begin{itemize}
\item[(a)]
If $\mcc$ contains $\Lambda$, then the following are equivalent.
\begin{itemize}
\item[(i)]
$\mcc$ is of finite representation type.
\item[(ii)]
The depth of the composition $l_S\pi_S$ in $\mcc$ is finite for every simple $\Lambda$-module $S$.\item[(iii)]
For any indecomposable projective $\Lambda$-module $P$, $\rad^\infty_{\mcc}(P,-)=0$.
\end{itemize}
\item[(b)]
If $\mcc$ contains $D\Lambda$, then the following are equivalent.
\begin{itemize}
\item[(i)]
$\mcc$ is of finite representation type.
\item[(ii)]
The depth of the composition $\iota_Sr_S$ in $\mcc$ is finite for every simple $\Lambda$-module $S$.
\item[(iii)]
For any indecomposable injective $\Lambda$-module $I$, $\rad^\infty_{\mcc}(-,I)=0$.
\end{itemize}
\end{itemize}
\end{theorem e*}
By using the depth of the compositions of the projective covers and the injective envelopes of simple $\Lambda$-modules, we find another nilpotency index of the radical of $\mcc$ which is independent from the maximal length of indecomposable $\Lambda$-modules in $\mcc$. More precisely, we assume $\theta_S\coloneqq\iota_S\pi_S$, for every simple $\Lambda$-module $S$ and prove the following theorem.
\begin{theorem f*}$($Theorem \ref{dp2}$)$\label{F}
Let $\mcc$ be a functorially finite subcategory of $\modd\Lambda$ which contains $\Lambda$ and $D\Lambda$. Then $\mcc$ is of finite representation type if and only if the depth of $\theta_S$ in $\mcc$ is finite, for every simple $\Lambda$-module $S$. Moreover, in this case, if $m$ is the maximal depth of the morphisms $\theta_S$ in $\mcc$ with $S$ ranging over all simple $\Lambda$-modules, then $m+1$ is a nilpotency index of $\rad_{\mcc}$.
\end{theorem f*}
Finally, we show that $n$-cluster tilting subcategories and the contravariantly finite resolving subcategories satisfy the above theorems. Also, when $\Lambda$ is virtually Gorenstein, the subcategory of finitely generated Gorenstein projective modules of $\modd\Lambda$ is another nice class of the functorially finite subcategories which satisfy the conditions of our main results.

The paper is organized as follows. In section 2, we recall some fundamental concepts and results that we need throughout this paper. Section 3 is devoted to the proof of theorems A, C and D. We prove theorems E and F in section 4. Finally, in section 5, we give some important examples of the functorially finite subcategories of $\modd\Lambda$ which satisfy the conditions of our main results.
\subsection{Notation}
Throughout this paper, we assume that $R$ is a commutative artinian ring and $\Lambda$ is an artin $R$-algebra.
We denote the category of left $\Lambda$-modules by $\Mod\Lambda$ and the category of finitely
generated left $\Lambda$-modules by $\modd\Lambda$. Also, we denote by $\mathrm{Proj}(\Lambda)$ (resp. $\mathrm{Inj}(\Lambda)$) and $\mathrm{proj}(\Lambda)$ (resp. $\mathrm{inj}(\Lambda)$) the full subcategories of $\Mod\Lambda$ consisting of projective (resp. injective) objects and finitely generated projective (resp. injective) objects, respectively. Let $J_R$ be the Jacobson radical of $R$. We denote by $D$ the duality $\Hom_R(-,I)$, where $I$ is the injective envelope of the $R$-module $R/J_R$. For a category $\mcc$, we denote by $\ind\mcc$ the full subcategory of $\mcc$ consisting of indecomposable objects and denote by $\mathrm{add}(\mcc)$ the full subcategory of $\mcc$ whose objects are direct summands of finite direct sums of objects in $\mcc$. We suppose that all subcategories of $\modd\Lambda$ are full and additive.
\section{Preliminaries}
In this section, we collect some basic definitions and facts that we need throughout this paper. First, we recall some notions and facts from the Auslander--Reiten theory. Next, we recall the concept of the radical of an additive category and some related results.
\begin{definition}$($\cite[Section 2]{AR4}$)$Let $\mathcal{C}$ be an arbitrary category.\begin{itemize}\item[(1)]A morphism $g: X\longrightarrow Y$ in $\mathcal{C}$ is said to be {\em right minimal} if any morphism $h: X\longrightarrow X$ such that $gh=g$ is an isomorphism.\item[(2)]A morphism $g: X\longrightarrow Y$ in $\mathcal{C}$ is said to be {\em right almost split} if $g$ is not a retraction and given any morphism $f:Z\longrightarrow Y$ which is not a retraction, there is a morphism $h:Z\longrightarrow X$ such that $gh=f$.\item[(3)]A morphism $g:X\longrightarrow Y$ in $\mathcal{C}$ is said to be {\em minimal right almost split} if it is both right minimal and right almost split.\item[(4)]A morphism $g:X\longrightarrow Y$ in $\mathcal{C}$ is said to be {\em irreducible} if\begin{itemize}\item[$(a)$]$g$ is neither a retraction nor a section and\item[$(b)$]given a commutative diagram\begin{center}\scalebox{.75}{\begin{tikzpicture}\node (X1) at (0,0) {$X$};\node (X2) at (4,0) {$Y$};\node (X3) at (2,-2) {$Z$};\draw [->,thick] (X1) -- (X2) node [midway,above] {$g$};\draw [->,thick] (X1) -- (X3) node [midway,left] {$f\,\, $};\draw [->,thick] (X3) -- (X2) node [midway,right] {$h$};\end{tikzpicture}}\end{center}either $h$ is a retraction or $f$ is a section.\end{itemize}\end{itemize}\end{definition}
Dually, the concepts of {\em left minimal}, {\em left almost split} and {\em minimal left almost split} morphisms are defined \cite{AR4}.

The category $\mcc$ has (minimal) right almost split morphisms, if for any indecomposable object $Y$ in $\mcc$ there is a (minimal) right almost split morphism $g:X\longrightarrow Y$ in $\mcc$. Dually, $\mcc$ has (minimal) left almost split morphisms, if for any indecomposable object $X$ in $\mcc$ there is a (minimal) left almost split morphism $f:X\longrightarrow Y$ in $\mcc$.

We recall that $\modd\Lambda$ is a Krull--Schmidt category \cite[II. Theorem 2.2]{ARS}. Let $\mcc$ is a subcategory of $\modd\Lambda$ which is closed under direct summands. It is obvious that $\mcc$ also is a Krull--Schmidt category.

A morphism $e:A\longrightarrow A$ in the category $\mathcal{C}$ is called {\em idempotent} if $e^2=e$ and $\mathcal{C}$ is called an {\em idempotent complete} category if for every idempotent $e:A\longrightarrow A$ in $\mathcal{C}$ there exist an object $B$ and morphisms $r:A\longrightarrow B$ and $s:B\longrightarrow A$ such that $sr=e$ and $rs=1_B$ \cite[Page 61]{F}. It is obvious that if $\mcc$ be a subcategory of $\modd\Lambda$ which is closed under isomorphisms and direct summands, then $\mcc$ is idempotent complete. Indeed, if $e : A\to A$ is an idempotent morphism in $\mcc$, then there exist a $\Lambda$-module $B$ and morphisms $r : A \to B$ and $s : B \to A$ such that $sr = e$ and $rs = 1_B$, since $\modd\Lambda$ is idempotent complete. Therefore, the morphism $s$ is a section and we have $A\cong B\oplus\Coker(s)$. Since $A \in\mcc$ and $\mcc$ is closed under isomorphisms and direct summands, then also $B\in\mcc$. Thus, $\mcc$ is idempotent complete.

We recall that an additive category $\mathcal{C}$ is said to be a \textit{variety} if it is essentially small idempotent complete \cite[Page 188]{Au1}.

Auslander and Reiten illustrated the connection between irreducible morphisms and minimal right and left almost split morphisms in \cite{AR4}.
\begin{theorem}$($\cite[Theorem 2.4]{AR4}$)$\label{irr-as}
Let $\mathcal{C}$ be an idempotent complete additive category and $C$ be an indecomposable object in $\mathcal{C}$.
\begin{itemize}
\item[$(a)$]
Assume that there is some $g^{\prime\prime}:B^{\prime\prime}\longrightarrow C$ which is a minimal right almost split morphism. Then $g:B\longrightarrow C$ is irreducible if and only if $B$ is not zero and there is some morphism $g^\prime:B^\prime\longrightarrow C$ such that the induced morphism $(g, g^\prime): B \coprod B'\longrightarrow C$ is minimal right almost split.
\item[(b)]
Assume that there is some $f^{\prime\prime}: C\longrightarrow D^{\prime\prime}$ which is a minimal left almost split morphism. Then $f: C\longrightarrow D$ is irreducible if and only if $D$ is not zero and there is some morphism $f': C\longrightarrow D'$ such that the induced morphism $\left[\begin{array}{ll}f\\f^\prime\end{array}\right]:C\longrightarrow D\prod D'$ is
minimal left almost split.
\end{itemize}
\end{theorem}
Let $\mathcal{C}$ be a subcategory of $\modd\Lambda$ and $X\in\modd\Lambda$. A morphism $\alpha:X\longrightarrow D$ is called a {\em left $\mathcal{C}$-approximation} if $D\in\mathcal{C}$ and for every $D^\prime\in\mathcal{C}$, the sequence of abelian groups $$\Hom_\Lambda(D,D^\prime)\overset{(\alpha,D^\prime)}{\longrightarrow}\Hom_\Lambda(X,D^\prime)\longrightarrow0$$ is exact. $\mathcal{C}$ is said to be a {\em covariantly finite subcategory} of $\modd\Lambda$ if for each $X\in\modd\Lambda$ there exists a left $\mathcal{C}$-approximation. Dually, the concepts of a {\em right $\mathcal{C}$-approximation} and a {\em contravariantly finite subcategory} are defined. Finally, $\mathcal{C}$ is called a {\em functorially finite subcategory} of $\modd\Lambda$ if it is both covariantly finite and contravariantly finite \cite[Page 81]{AS}.

The subcategory $\mcc$ of $\modd\Lambda$ is called of finite representation type if the number of isomorphism classes of indecomposable objects in $\mcc$ is finite.

\begin{remark}\label{mlra}
Let $\mcc$ be a contravariantly finite subcategory of $\modd\Lambda$ which is closed under direct summands. For every $\Lambda$-module $Y$ in $\mcc$, there exists a minimal right almost split morphism $g:X\longrightarrow Y$ in $\mcc$. Dually, if $\mcc$ is a covariantly finite subcategory of $\modd\Lambda$ which is closed under direct summands, then for every $\Lambda$-module $X$ in $\mcc$ there exists a minimal left almost split morphism $f:X\longrightarrow Y$ in $\mcc$ (see, \cite[Proposition 3.10]{AS}).\end{remark}
The radical of an additive category was defined for the first time by Kelly in \cite{K}. For more information about this concept, the readers are referred to \cite{S} and \cite[Appendix A.3]{ASS}.

Let $\mathcal{C}$ be an additive category. The \textit{(Jacobson) radical} of $\mathcal{C}$ is the two-sided ideal $\rad_\mathcal{C}$ in $\mathcal{C}$ defined by the formula
$$\rad_\mathcal{C}(X,Y)=\{h\in\Hom_\mathcal{C}(X,Y) \mid 1_X-gh\;\text{is invertible for any}\;g\in\Hom_\mathcal{C}(Y,X) \},$$
for all objects $X$ and $Y$ in $\mathcal{C}$.

For any $i\geq1$, the $i$-th power of $\rad_\mathcal{C}(X,Y)$ consists all finite sums of morphisms of the form
\begin{equation}
X= X_0 \overset{h_1}{\longrightarrow} X_1 \overset{h_2}{\longrightarrow}X_2\longrightarrow\cdots \longrightarrow X_{i-1} \overset{h_i}{\longrightarrow} X_i=Y \notag,
\end{equation}
where $h_j\in\rad_\mathcal{C}(X_{j-1},X_j)$, for every $1\leq j\leq i$. Thus, for any $i\geq1$ the $i$-th power $\rad_\mathcal{C}^i$ of $\rad_\mathcal{C}$ is defined by considering all morphisms in $\rad^i_\mathcal{C}(X,Y)$, for all objects $X$ and $Y$ of $\mathcal{C}$. Also, the \textit{infinite radical of $\mathcal{C}$} denoted by $\rad_\mathcal{C}^\infty$ is defined as the intersection of all powers $\rad_\mathcal{C}^i$, for $i\geq 1$. Indeed,$$\rad_\mathcal{C}^\infty\coloneqq\bigcap_{i\geq1}\rad_\mathcal{C}^i.$$
We denote by $\rad_\Lambda$ the radical of $\modd\Lambda$.

Let $\mcc$ be a full subcategory of $\modd\Lambda$. It is obvious that $\rad_{\mcc}(X,Y)=\rad_\Lambda(X,Y)$, for $\Lambda$-modules $X$ and $Y$ in $\mcc$. Thus, some facts that have been stated around $\rad_\Lambda$ are practical for $\rad_{\mcc}$. Specially, if $X$ is an indecomposable $\Lambda$-module in $\mcc$, then $\rad_{\mcc}(X,X)$ is the set of all non-isomorphisms on $X$.

We recall that the following properties are satisfied for the radical of a subcategory $\mathcal{C}$ of $\modd\Lambda$ which is closed under isomorphisms.
\begin{lemma}$($\cite[Page 80]{AS}$)$\label{lemrad}
Let $\mathcal{C}$ be a full subcategory of $\modd\Lambda$ that is closed under isomorphisms.
\begin{itemize}
\item[$(a)$]
Assume that $X$ and $Y$ are two objects of $\mathcal{C}$. $\rad_\mathcal{C}(X,Y)=\Hom_\mathcal{C}(X,Y)$ if and only if no indecomposable summand of $X$ is isomorphic to an indecomposable summand of $Y$.
\item[$(b)$]
Assume that $Y$ is an indecomposable object of $\mathcal{C}$. For a morphism $f:X\longrightarrow Y$ in $\mathcal{C}$ the following are equivalent:
\begin{itemize}
\item[$(1)$]
$f$ is not a retraction,
\item[$(2)$]
$f$ is in $\rad_\mathcal{C}(X, Y)$.
\end{itemize}
\item[$(c)$]
Assume that $Y$ is an indecomposable object of $\mathcal{C}$. For a morphism $g:Y\longrightarrow Z$ in $\mathcal{C}$ the following are equivalent:
\begin{itemize}
\item[$(1)$]
$g$ is not a section,
\item[$(2)$]
$g$ is in $\rad_\mathcal{C}(Y, Z)$.
\end{itemize}
\end{itemize}
\end{lemma}

Let $X$ and $Y$ be indecomposable $\Lambda$-modules. The relation between irreducible morphisms in $\modd\Lambda$ and
the radical of $\modd\Lambda$ has already been studied. In fact, a morphism $f:X\longrightarrow Y$ is irreducible if and only if $f\in\rad_\Lambda(X,Y)\backslash\rad_\Lambda^2(X,Y)$ \cite[Proposition 1.(a)]{B}.

In the following lemma, we show the fact above for a subcategory of $\modd\Lambda$ that is closed under isomorphisms and direct summands. The proof is similar to the proof of the case $\modd\Lambda$  (for example see \cite[IV. Lemma 1.6]{ASS}).
\begin{lemma}\label{irr-rad}
Let $\mathcal{C}$ be a subcategory of $\modd\Lambda$ that is closed under isomorphisms and direct summands and $X$ and $Y$ be two indecomposable objects in $\mcc$. A morphism $f:X\longrightarrow Y$ is irreducible in $\mcc$ if and only if $f\in\rad_{\mcc}(X,Y)\backslash\rad_{\mcc}^2(X,Y)$.
\end{lemma}
\section{The infinite radical and the representation type of functorially finite subcategories}
In this section, we give the criteria to describe whether a functorially finite subcategory $\mcc$ of $\modd\Lambda$ is of finite representation type. In fact, we use the concept of the infinite radical of $\mcc$ and show that a functorially finite subcategory $\mcc$ of $\modd\Lambda$ which contains $D\Lambda$ or $\Lambda$ is of finite representation type if and only if $\rad^\infty_{\mcc}$ vanishes. Also, in this case we describe the morphisms in powers of the radical of $\mcc$ in terms of irreducible morphisms in $\mcc$. Moreover, we show that a functorially finite subcategory $\mcc$ of $\modd\Lambda$ that satisfies the specific condition is of finite representation type if and only if any family of monomorphisms (resp. epimorphisms) between indecomposable objects in $\mathcal{C}$ is noetherian (resp. conoetherian). In the rest of the paper, we assume that all subcategories of $\modd\Lambda$ are closed under direct summands and isomorphisms. It is obvious that these subcategories are the Krull--Schmidt variety.
	
First, we recall some basic notions and facts on functor categories and refer readers to \cite{Au,Au1,Au2}, for more details. Let $\mathcal{C}$ be an essentially small additive category. We denote by $\Mod(\mathcal{C})$ the category of all additive covariant functors from $\mathcal{C}$ to the category Ab consisting of all abelian groups. Every object of $\Mod(\mathcal{C})$ is called {\em $\mathcal{C}$-module} and morphisms of $\Mod(\mathcal{C})$ are all natural transformations of $\mathcal{C}$-modules. For a $\mathcal{C}$-module $F$, if there exists an epimorphism $$\coprod_{i\in I}\Hom_\mathcal{C}(C_i,-)\longrightarrow F\longrightarrow0,$$with $C_i$ in $\mathcal{C}$ and $I$ a finite set, then $F$ is called {\em finitely generated} and if there exists an exact sequence$$\coprod_{j\in J}\Hom_\mathcal{C}(C_j,-)\longrightarrow\coprod_{i\in I}\Hom_\mathcal{C}(C_i,-)\rightarrow F\rightarrow 0,$$with $C_i$ and $C_j$ in $\mathcal{C}$ and $I$ and $J$ finite sets, then $F$ is called {\em finitely presented}.

A $\mathcal{C}$-module $F^\prime$ is a {\em submodule} of the $\mathcal{C}$-module $F$, if there is a morphism $g:F^\prime\longrightarrow F$ such that $g_C:F^\prime(C)\longrightarrow F(C)$ is an inclusion morphism of abelian groups for each $C\in\mathcal{C}$.

A $\mathcal{C}$-module $F$ is called {\em noetherian} (resp., {\em artinian}) if it satisfies the ascending (resp., descending) chain condition on submodules and if $F$ is both noetherian and artinian then it is called {\em finite}.

Also, a $\mathcal{C}$-module $F$ is said to be {\em locally finite} if every finitely generated submodule of $F$ is finite and the category $\Mod(\mathcal{C})$ is said to be {\em locally finite} if every $\mathcal{C}$-module is locally finite. A $\mathcal{C}$-module $F$ is called {\em simple} if it is not zero and only its submodules are $0$ and $F$.

Let $\mathcal{C}$ be a Krull--Schmidt variety and $F$ be a $\mathcal{C}$-module.
Suppose that $X$ is an object in $\mathcal{C}$. An element $x$ in $F(X)$ is said to be a {\em minimal element} if $x\neq0$ and has the property that given any proper epimorphism $f:X\longrightarrow X^\prime$, i.e., an epimorphism that is not an isomorphism, then $F(f)(x)=0$. An element $x$ in $F(X)$ is said to be a {\em universally minimal element} if $x\neq0$ and has the property that given any $f:X\longrightarrow X^\prime$ which is not a section, then $F(f)(x)=0$.

\begin{theorem}\label{th0}
$($\cite[Theorem 2.10]{Au2}$)$  A Krull-Schmidt variety $\mcc$ has the property that $\Mod(\mathcal{C})$ is locally
finite if and only if
\begin{itemize}
\item[(i)]
Given any indecomposable object $C\in\mcc$ there is a left almost split morphism $f:C \to B$.
\item[(ii)]
 If $F$ is a nonzero $\mathcal{C}$-module, then there is an object $C\in\mathcal{C}$ such that $F(C)$ contains a universally minimal element.
\end{itemize}
\end{theorem}
 We have the following result for the functorially finite subcategories of $\modd\Lambda$.
 \begin{proposition}$($\cite[Corollary 4.3]{FN}$)$\label{loc}
 Let $\mcc$ be a functorially finite subcategory of $\modd\Lambda$. $\Mod(\mathcal{C})$ is locally finite if and only if $\mathcal{C}$ is of finite representation type.
 \end{proposition}
We need the following crucial lemma in the rest of this section.
\begin{lemma}\label{lem0}
Let $\mcc$ be a subcategory of $\modd\Lambda$ and $F$ be an $\mcc$-module. The following statements hold.
\begin{itemize}
\item[(i)] If $x$ in $F(X)$ is a minimal element, then $X$ is an indecomposable $\Lambda$-module.
\item[(ii)] If $X\in\mcc$ and $x$ in $F(X)$ is not zero, then there is an epimorphism $f:X\longrightarrow X^\prime$ in $\mcc$ such that $F(f)(x)$ is a minimal element in $F(X^\prime)$.
\end{itemize}
\end{lemma}
\begin{proof}
$($a$)$: Assume that $X$ is not indecomposable, so there is a decomposition $X=X_1\oplus X_2$ with $X_1$ and $X_2$ nonzero $\Lambda$-modules. Since $\mcc$ is closed under direct summands, $X_1,X_2\in\mcc$. Consider the natural projection $p_j:X\longrightarrow X_j$ in $\mcc$, for $j=1,2$. We know that $1_X=i_1p_1+i_2p_2$, where $i_1$ and $i_2$ are natural inclusions. We apply $F$ and obtain
\begin{equation}\label{eq0}
1_{F(X)}=F(1_X)=F(i_1)F(p_1)+F(i_2)F(p_2).
\end{equation}
Since $x$ in $F(X)$ is a minimal element and $p_1$ and $p_2$ are proper epimorphisms in $\mcc$, we have $F(p_1)(x)=F(p_2)(x)=0$. Therefore, by putting $x$ in the equality \eqref{eq0}, we have $x=1_{F(X)}(x)=0$. But this is a contradiction.

$($b$)$: Assume that $\pi_{K}$ is the canonical epimorphism $X\longrightarrow X/K$, for every submodule $K$ of $X$. Set $$\mathcal{S}\coloneqq \{ K \text{\, is a submodule of\,\,} X\,|\, X/K\in\mcc \text{\, and\,\,} F(\pi_K)(x)\neq0 \}.$$
Zero object belongs to $\mathcal{S}$. Therefore, $\mathcal{S}$ is not empty and by \cite[Proposition 10.9]{AF}, $\mathcal{S}$ has a maximal element such as $Y$. Since $Y\in\mathcal{S}$, $X/Y\in\mcc$ and the canonical epimorphism $\pi_Y:X\longrightarrow X/Y$ lies in $\mcc$. We claim that $F(\pi_Y)(x)$ is a minimal element in $F(X/Y)$. It is clear that $F(\pi_ Y)(x)$ is nonzero. Consider a proper epimorphism $g:X/Y\longrightarrow Z$ in $\mcc$. $\mathrm{Ker}(g)$ is nonzero and a submodule of $X/Y$. Then there exists a submodule $W$ of $X$ such that $Y$ is a submodule of $W$ and $\mathrm{Ker}(g)\cong W/Y$ and we have $Z\cong X/W.$ Since $\mcc$ is closed under isomorphisms, $X/W$ lies in $\mcc$. Now, suppose that $F(g)F(\pi_Y)(x)$ is nonzero. Therefore, $F(\pi_{W})(x)$ is nonzero and so $W$ belongs to $\mathcal{S}$. But this contradicts the maximality of $Y$. Hence, $F(g)F(\pi_Y)(x)=0$.
\end{proof}
The above lemma has already been proved for the category $\Mod\Lambda$, by Auslander (see \cite[Proposition 1.4]{AuL}). Now, we recall the definition of the noetherian and conoetherian family of morphisms from \cite[Section 3]{Au2}.

\begin{definition}
Let $\mathcal{C}$ be an additive category. A family $\{f_j\}_{j\in J}$ of morphisms in $\mathcal{C}$ is called {\em noetherian} if given any chain
\begin{equation}
C_0 \overset{f_{j_1}}{\longrightarrow} C_{1}\overset{f_{j_2}}{\longrightarrow} \cdots \overset{f_{j_{i-1}}}{\longrightarrow}C_{i-1}\overset{f_{j_i}}{\longrightarrow} C_{i}\overset{f_{j_{i+1}}}{\longrightarrow} \cdots \notag
\end{equation}
of morphisms in $\{f_j\}_{j\in J}$ such that $f_{j_i}\cdots f_{j_2}f_{j_1} \neq0$ for all $j_i$, there is an integer $d\geq 1$ such that $f_{j_k}$ is an isomorphism for all $k\geq d$. A family $\{f_j\}_{j\in J}$ of morphisms in $\mathcal{C}$ is called {\em conoetherian} if given any chain
\begin{equation}
\cdots\overset{f_{j_{i+1}}}{\longrightarrow}C_{i} \overset{f_{{j_i}}}{\longrightarrow} C_{i-1}\overset{f_{j_{i-1}}}{\longrightarrow} \cdots {\longrightarrow}C_1\overset{f_{j_1}}{\longrightarrow} C_{0} \notag\end{equation}of morphisms in $\{f_j\}_{j\in J}$ such that $f_{j_1}\cdots f_{j_{i-1}}f_{j_i}\neq0$ for all $j_i$, there is an integer $d\geq 1$ such that $f_{j_k}$ is an isomorphism for all $k\geq d$.
\end{definition}
According to the Auslander's works, we have $\modd\Lambda$ is of finite representation type if and only if every family of monomorphisms (resp. epimorphisms) between indecomposable objects in $\modd\Lambda$ is noetherian (resp. conoetherian) (for example see \cite[Theorem 3.1]{Au2} and \cite[Theorem A]{AuL}). It is natural to ask when these results are satisfied for a functorially finite subcategory $\mathcal{C}$ of $\modd\Lambda$. If $\mathcal{C}$ is of finite representation type, then it is easy to see that any family of monomorphisms (resp. epimorphisms) between indecomposable objects in $\mathcal{C}$ is noetherian (in this case, we do not need the functorially finiteness condition). To prove the other direction, we need the following new definition.
\begin{definition}\label{sm}
Let $\mathcal{C}$ be a Krull--Schmidt variety and $F$ be a $\mathcal{C}$-module. Suppose that $X$ is an object in $\mathcal{C}$. An element $x$ in $F(X)$ is said to be a {\em strongly minimal element} if $x\neq0$ and has the property that given any $f:X\longrightarrow X^\prime$ which is not a monomorphism, then $F(f)(x)=0$.
\end{definition}
It is obvious that any universally minimal element in $F(X)$ is strongly minimal in $F(X)$ and any strongly minimal element in $F(X)$ is minimal in $F(X)$. Specially, suppose that $F$ is a $(\modd\Lambda)$-module and $X\in\modd\Lambda$. It is easy to see that if $x$ is minimal in $F(X)$ then it is also strongly minimal in $F(X)$. Therefore, the minimal elements are exactly strongly minimal elements, for $(\modd\Lambda)$-modules. In fact, we have the following result.

\begin{proposition}$($\cite[Proposition 1.4]{AuL}$)$\label{pro2}
 Let $F$ be an arbitrary $(\modd\Lambda)$-module and $X$ be an object in $\modd\Lambda$. A nonzero element $x$ is minimal in $F(X)$ if and only if a morphism $f : X\to X^\prime$ in $\modd\Lambda$ is a monomorphism whenever it has the property $F(f)(x)$ in $F(X^\prime)$ is not zero.
\end{proposition}
\begin{remark}\label{r1}
It is easy to see that Proposition \ref{pro2} is satisfied for any subcategory of $\modd\Lambda$ which is closed under the image of morphisms. In fact, if $\mcc$ is a subcategory of $\modd\Lambda$ which is closed under the image of morphisms and $X$ is an object in $\mathcal{C}$, then a nonzero element $x$ is minimal in $F(X)$ if and only if it is strongly minimal in $F(X)$.
\end{remark}
 As we mentioned before, Auslander proved that if every family of monomorphisms (resp. epimorphisms) between indecomposable objects in $\modd\Lambda$ is noetherian (resp. conoetherian), then $\modd\Lambda$ is of finite representation type. For the proof of these results, he used Lemma \ref{lem0} that is always satisfied for $\modd\Lambda$. Now, consider a functorially finite subcategory $\mcc$ of $\modd\Lambda$. For proving these results for $\mcc$, we can not use the Auslander's method. Instead, we use the following condition.
\begin{condition}\label{c1}
For any $\mcc$-module $F$, $X\in\mcc$ and $0\neq x\in F(X)$, there is a morphism $f:X\longrightarrow X^\prime$ in $\mcc$ such that $F(f)(x)$ is a strongly minimal element in $F(X^\prime)$.
\end{condition}
 \begin{proposition}\label{pro3}
 Let $\mcc$ be a functorially finite subcategory of $\modd\Lambda$. If Condition \ref{c1} is satisfied and any family of monomorphisms between indecomposable objects in $\mathcal{C}$ is noetherian, then $\mathcal{C}$ is of finite representation type.
 \end{proposition}
 \begin{proof}
 Assume that $\mcc$ is not of finite representation type. By Proposition \ref{loc}, $\Mod(\mathcal{C})$ is not locally finite. By Remark \ref{mlra}, $\mcc$ has left almost split morphisms, so Theorem \ref{th0} implies that there exists a nonzero $\mcc$-module $F$ such that for any $X\in\mcc$, $F(X)$ has no universally minimal elements. More precisely, for any $X\in\mcc$ and nonzero element $x\in F(X)$ there exists a morphism $g:X\longrightarrow Y$ in $\mcc$ such that is not a section and $F(g)(x)\neq 0$ in $F(Y)$. We claim that there exists an infinite sequence of proper monomorphisms between indecomposable $\Lambda$-modules in $\mcc$, which gives a contradiction. $F$ is nonzero, so we can fix an $X\in\mcc$ such that $F(X)\neq0$. Consider a nonzero element $x\in F(X)$. By the assumption, there is a morphism $f_0:X\longrightarrow X_0$ in $\mcc$ such that $x_0\coloneqq F(f_0)(x)$ is strongly minimal in $F(X_0)$. Clearly, $x_0$ is nonzero and since $F(X_0)$ has no universally minimal elements, there exists a morphism $g_0:X_0\longrightarrow Y_0$ in $\mcc$ such that is not a section and $F(g_0)(x_0)\neq0$ in $F(Y_0)$. Consider $Y_0\in \mcc$ and the nonzero element $F(g_0)(x_0)$ in $F(Y_0)$. We use the assumption again and conclude that there exists a morphism $h_0:Y_0\longrightarrow X_1$ in $\mcc$ such that $x_1\coloneqq F(h_0)F(g_0)(x_0)$ is a strongly minimal element in $F(X_1)$. Set $f_1\coloneqq h_0g_0$. If $f_1$ is not a monomorphism, then $x_1=F(f_1)(x_0)=0$ since $x_0$ is a strongly minimal element in $F(X_0)$. But this is a contradiction since $x_1$ is nonzero. Thus, $f_1$ is a monomorphism. If $f_1$ is not a proper monomorphism, then $f_1$ is an isomorphism. Therefore, there exists a left inverse for $f_1$ and so $g_0$ is a section that is a contradiction. Thus, $f_1$ is a proper monomorphism. Until here, we have the following diagram
  \begin{center}
   \begin{tikzpicture}
    \node (X1) at (0,1) {$X$};
    \node (X2) at (3,1) {$X_0$};
    \node (X3) at (6,1) {$X_1,$};
    \node (X4) at (3,-1) {$Y_0$};
    \draw [->,thick] (X1) -- (X2) node [midway,above] {$f_0$};
    \draw [->,thick] (X2) -- (X3) node [midway,above] {$f_1$};
    \draw [->,thick,dashed] (X2) -- (X4) node [midway,left] {$g_0$};
    \draw [->,thick,dashed] (X4) -- (X3) node [midway,right] {$\quad h_0$};\end{tikzpicture}
    \end{center}
    that $f_1$ is a proper monomorphism and $x_0=F(f_0)(x)$ and $x_1=F(f_1f_0)(x)$ are strongly minimal elements in $F(X_0)$ and $F(X_1)$, respectively. Assume inductively that we obtain a sequence
    \begin{equation}
    X\overset{f_0}{\longrightarrow}X_0 \overset{f_1}{\longrightarrow} X_{1}{\longrightarrow} \cdots \overset{f_{k-1}}{\longrightarrow} X_{k-1}\overset{f_k}{\longrightarrow} X_k. \notag
    \end{equation}
    in $\mcc$, for $k\geq1$ such that $f_i:X_{i-1}\longrightarrow X_i$ is a proper monomorphism for each $1\leq i\leq k$ and $x_i\coloneqq F(f_i\cdots f_0)(x)$ is a strongly minimal element in $F(X_i)$ for all $0\leq i\leq k$. Since $x_k$ is a strongly minimal element in $F(X_k)$, it is nonzero. By assumption, $F$ has no universally minimal elements, so there exists a morphism $g_k:X_k\longrightarrow Y_k$ in $\mcc$ such that is not a section and $F(g_k)(x_k)\neq0$ in $F(Y_k)$. Hence, by the assumption, there exists a morphism $h_k:Y_k\longrightarrow X_{k+1}$ in $\mcc$ such that $x_{k+1}\coloneqq F(h_k)F(g_k)(x_k)$ is strongly minimal in $F(X_{k+1})$. Set $f_{k+1}\coloneqq h_kg_k$. Therefore, we have $x_{k+1}=F(h_k)F(g_k)(x_k)=F(f_{k+1}f_k\cdots f_0)(x)$ that is strongly minimal in $F(X_{k+1})$. Also $f_{k+1}$ is a proper monomorphism, by using a similar argument applied in the first step of the induction. Thus, we can construct the infinite sequence
    \begin{equation}
    X \overset{f_0}{\longrightarrow} X_{0}\overset{f_1}{\longrightarrow} X_1 \overset{f_2}{\longrightarrow}X_2\overset{f_3}{\longrightarrow} X_3\overset{f_4}{\longrightarrow} \cdots \notag
    \end{equation}
    in $\mcc$ such that $f_i$ is a proper monomorphism, for any $i\geq1$. For any $i\geq0$, $x_i$ is a strongly minimal element in $F(X_i)$ and consequently $x_i$ is a minimal element in $F(X_i)$. Thus, by Lemma \ref{lem0}, $X_i$ is an indecomposable $\Lambda$-module for any $i\geq0$. Therefore, we have the sequence
    \begin{equation}
    X_{0}\overset{f_1}{\longrightarrow} X_1 \overset{f_2}{\longrightarrow}X_2\overset{f_3}{\longrightarrow} X_3\overset{f_4}{\longrightarrow} \cdots\notag
    \end{equation}
     between indecomposable $\Lambda$-modules in $\mcc$ which is the same desired sequence.
  \end{proof}
 In Theorem \ref{th11}, we will show that the converse of Proposition \ref{pro3} is true provided that $\mcc$ contains $\Lambda$ or $D\Lambda$.
 
 If we consider additive contravariant functors $F$ from $\mathcal{C}$ to Ab and define the suitable version of strongly minimal elements, then by duality, we can prove Proposition \ref{pro3}, for the case that any family of epimorphisms between indecomposable objects in $\mathcal{C}$ is conoetherian.

  Remark \ref{r1} and Lemma \ref{lem0} imply that Condition \ref{c1} is satisfied for a functorially finite subcategory  $\mcc$ of $\modd\Lambda$ which is closed under the image of morphisms. Therefore, we have the following result.
  \begin{corollary}\label{cor1}
 Let $\mcc$ be a functorially finite subcategory of $\modd\Lambda$ which is closed under the image of morphisms. If any family of monomorphisms (resp. epimorphisms) between indecomposable objects in $\mathcal{C}$ is noetherian (resp. conoetherian), then $\mathcal{C}$ is of finite representation type. Specially, if any family of monomorphisms (resp. epimorphisms) between indecomposable $\Lambda$-modules is noetherian (resp. conoetherian), then $\modd\Lambda$ is of finite representation type.
 \end{corollary}

Assume that $\mcc$ is an arbitrary subcategory of $\modd\Lambda$. As we mentioned before, if $\mathcal{C}$ is of finite representation type, then it is easy to see that any family of monomorphisms (resp. epimorphisms) between indecomposable objects in $\mathcal{C}$ is noetherian (resp. conoetherian).
 \begin{proposition}\label{pro5}
 Let $\mcc$ be an arbitrary subcategory of $\modd\Lambda$. If $\mathcal{C}$ is of finite representation type, then any family of monomorphisms (resp. epimorphisms) between indecomposable objects in $\mathcal{C}$ is noetherian (resp. conoetherian).
 \end{proposition}
 \begin{proof}
 Assume that there exists an infinite sequence
  \begin{equation}
 X_{0}\overset{f_1}{\longrightarrow} X_1 \overset{f_2}{\longrightarrow}X_2\overset{f_3}{\longrightarrow} X_3\overset{f_4}{\longrightarrow} \cdots \notag
 \end{equation}
  of proper monomorphisms between indecomposable $\Lambda$-modules in $\mcc$. $f_1$ is a monomorphism which is not an epimorphism, so the length of $X_0$ as $R$-module is strictly less than the length of $X_1$ as $R$-module. By repeating this process for any $f_i$, we conclude that there are infinitely many non-isomorphic indecomposable $\Lambda$-module in $\mcc$ and so $\mcc$ is not of finite representation type. By duality, if $\mathcal{C}$ is of finite representation type, then any family of epimorphisms between indecomposable objects in $\mathcal{C}$ is conoetherian.
  \end{proof}
  As a consequence of Proposition \ref{pro5} and Corollary \ref{cor1}, we have the following result.
 \begin{corollary}\label{cor2}
 Let $\mcc$ be a functorially finite subcategory of $\modd\Lambda$ which is closed under the image of morphisms. Then any family of monomorphisms (resp. epimorphisms) between indecomposable objects in $\mathcal{C}$ is noetherian (resp. conoetherian) if and only if $\mathcal{C}$ is of finite representation type. Specially, any family of monomorphisms (resp. epimorphisms) between indecomposable $\Lambda$-modules is noetherian (resp. conoetherian) if and only if $\modd\Lambda$ is of finite representation type.
\end{corollary}
Let $\mcc$ be a functorially finite subcategory of $\modd\Lambda$. By Lemma \ref{lem0}, for any $\mcc$-module $F$, $X\in\mcc$ and and $0\neq x\in F(X)$, there is an epimorphism $f:X\longrightarrow X^\prime$ in $\mcc$ such that $F(f)(x)$ is a minimal element in $F(X^\prime)$. Therefore, by using an argument similar to the proof of Proposition \ref{pro3}, we have the following result.
  \begin{proposition}\label{pro6}
  Let $\mcc$ be a functorially finite subcategory of $\modd\Lambda$. If any family of morphisms between indecomposable objects in $\mathcal{C}$ is noetherian, then $\mathcal{C}$ is of finite representation type.
  \end{proposition}
We recall the Harada--Sai lemma that was stated for the first time by Harada and Sai in \cite[Lemma 9]{HS}.
\begin{lemma}$($Harada-Sai Lemma$)($\cite[Page 115]{R}$)$
 Let $\Lambda$ be an artin algebra and $M_1,\cdots,M_{2^b}$ be indecomposable $\Lambda$-modules of length at most $b$. Suppose that $f_i:M_i\rightarrow M_{i+1}$ is a non-isomorphism, for $1\leq i\leq 2^b-1$. Then the composition $f_{2^b-1}\cdots f_2f_1$ is zero.
 \end{lemma}
Now, we are ready to prove the first main result in this section.
\begin{theorem}\label{main}
Let $\mcc$ be a functorially finite subcategory of $\modd\Lambda$.
\begin{itemize}
\item[(a)]
If $\mcc$ contains $D\Lambda$, then the following are equivalent.
\begin{itemize}
\item[(i)]
There exists $t\in\mathbb{N}$ such that for any $X\in\ind\mcc$, $\rad^t_{\mcc}(X,-)=0$.
\item[(ii)]
For any $X\in\ind\mcc$, there exists $t\in\mathbb{N}$ such that $\rad^t_{\mcc}(X,-)=0$.
\item[(iii)]
Any family of morphisms between indecomposable $\Lambda$-modules in $\mcc$ is noetherian.
\item[(iv)]
$\mcc$ is of finite representation type.
\item[(v)]
$\rad^\infty_{\mcc}=0$.
\end{itemize}
\item[(b)]
If $\mcc$ contains $\Lambda$, then the following are equivalent.
\begin{itemize}
\item[(i)]
There exists $t\in\mathbb{N}$ such that for any $X\in\ind\mcc$, $\rad^t_{\mcc}(-,X)=0$.
\item[(ii)]
For any $X\in\ind\mcc$, there exists $t\in\mathbb{N}$ such that $\rad^t_{\mcc}(-,X)=0$.
\item[(iii)]
Any family of morphisms between indecomposable $\Lambda$-modules in $\mcc$ is conoetherian.
\item[(iv)]
$\mcc$ is of finite representation type.
\item[(v)]
$\rad^\infty_{\mcc}=0$.
\end{itemize}
\end{itemize}
\end{theorem}
\begin{proof}
We only prove $(\mathrm{a})$, because the proof of $(\mathrm{b})$ is dual.

$(\mathrm{i})\Rightarrow(\mathrm{ii})$ It is clear.

$(\mathrm{ii})\Rightarrow(\mathrm{iii})$ Assume that there is a family $\{f_j\}_{j\in J}$ of morphisms between indecomposable $\Lambda$-modules in $\mcc$ such that is not noetherian. So there is an infinite sequence
\begin{equation}
X_0 \overset{f_1}{\longrightarrow} X_{1}\overset{f_2}{\longrightarrow} X_2 \overset{f_3}{\longrightarrow}X_3\overset{f_4}{\longrightarrow} \cdots \notag
\end{equation}
of morphisms between indecomposable $\Lambda$-modules in $\mcc$ such that for any $i\in\mathbb{N}$, $f_i\cdots f_1\neq0$ and $f_i$ is not an isomorphism.
Therefore, for any $i\in\mathbb{N}$, $f_i\dots f_1\in\rad^i_{\mcc}(X_0,X_i)$ and so $\rad^i_{\mcc}(X_0,-)\neq0$. This is a contradiction.

$(\mathrm{iii})\Rightarrow(\mathrm{iv})$ Follows by Proposition \ref{pro6}.

$(\mathrm{iv})\Rightarrow(\mathrm{v})$ Assume that $\mcc$ is of finite representation type, so the number of isomorphism classes of indecomposable $\Lambda$-modules in $\mcc$ is finite. Suppose that $b$ is the maximal length of indecomposable $\Lambda$-modules in $\mcc$. By Harada--Sai Lemma, $\rad^{2^b-1}_{\mcc}=0$ and therefore $\rad^\infty_{\mcc}=0$.

$(\mathrm{v})\Rightarrow(\mathrm{ii})$ Assume $\rad^\infty_{\mcc}=0$ and consider an arbitrary $X\in\ind\mcc$. By assumption, $D\Lambda\in\mcc$. So $\rad^\infty_{\mcc}(X,D\Lambda)=0$. It is obvious that $\Hom_{\mcc}(A,B)=\Hom_\Lambda(A,B)$ is of finite length, for every two $\Lambda$-modules $A,B\in\mcc$. Thus, the descending chain
$$\Hom_{\mcc}(A,B)\supseteq \rad_{\mcc}(A,B)\supseteq \rad_{\mcc}^2(A,B)\supseteq\rad_{\mcc}^3(A,B)\supseteq\cdots$$  is stationary, and hence there exists $t\in\mathbb{N}$ such that $\rad_{\mcc}^\infty(A,B)=\rad_{\mcc}^t(A,B)$. Specially, $\rad_{\mcc}^\infty(X,D\Lambda)=\rad_{\mcc}^t(X,D\Lambda)=0$ for some $t\in\mathbb{N}$. We claim that for every $Y\in\mcc$, $\rad_{\mcc}^t(X,Y)=0$. Assume that $g$ is a morphism in $\rad_{\mcc}^t(X,Y)$. Since $D\Lambda$ is an injective cogenerator, there is a monomorphism $f:Y\longrightarrow (D\Lambda)^k$. The composition $fg$ belongs to $\rad_{\mcc}^t(X,(D\Lambda)^k)$ and so $fg=0$ and consequently $g$ is zero because $f$ is a monomorphism. Therefore, $\rad_{\mcc}^t(X,Y)=0$.

$(\mathrm{ii})\Rightarrow(\mathrm{i})$ Assume that for any $X\in\ind\mcc$, there exists $t_X\in\mathbb{N}$ such that $\rad^{t_X}_{\mcc}(X,-)=0$.
By the previous parts of the proof, we conclude that $\mcc$ is of finite representation type. So the number of isomorphism classes of indecomposable $\Lambda$-modules in $\mcc$ is finite. Set $t=\max\{t_X\,|\, X\in\ind\mcc\}$. Hence, $\rad^t_{\mcc}(X,-)=0$ for any $X\in\ind\mcc$.
\end{proof}

Let $\mcc$ be a functorially finite subcategory of $\modd\Lambda$ which contains $\Lambda$ or $D\Lambda$. In the proof of Theorem \ref{main}, we observe that if $\mcc$ is of finite representation type then there exists an integer $d\geq1$ such that $\rad^d_{\mcc}=0$. Such an integer $d$ is called the {\em nilpotency index} of the radical of $\mcc$. Specially, $\rad^{2^b-1}_{\mcc}=0$ where $b$ is the maximal length of indecomposable $\Lambda$-modules in $\mcc$.
Due to Harada--Sai Lemma, the nilpotency index $2^b-1$ is called Harada--Sai bound.

We show that when a functorially finite subcategory $\mcc$ of $\modd\Lambda$ contains $\Lambda$ or $D\Lambda$ and it is of finite representation type, Condition \ref{c1} is satisfied.
\begin{proposition}\label{pro7}
Let $\mcc$ be a functorially finite subcategory of $\modd\Lambda$ which contains $\Lambda$ or $D\Lambda$. If $\mcc$ is of finite representation type, then for any $\mcc$-module $F$, $X\in\mcc$ and $0\neq x\in F(X)$, there is a morphism $f:X\longrightarrow X^\prime$ in $\mcc$ such that $F(f)(x)$ is a strongly minimal element in $F(X^\prime)$.
\end{proposition}
\begin{proof}
Assume that $F$ is a $\mcc$-module, $X\in\mcc$ and $0\neq x\in F(X)$. We consider two cases. Case 1: assume that for any $X^\prime\in\mcc$ and any morphism $h:X\to X^\prime$ which is not a monomorphism, $F(h)(x)=0$. We put $f\coloneqq 1_X$. It is obvious that $F(f)(x)$ is strongly minimal in $F(X)$.

\noindent
Case 2: if the case one is not satisfied, then there exist some $X^\prime\in\mcc$ and morphisms $h:X\to X^\prime$  which is not a monomorphism with the property $F(h)(x)\neq0$. More precisely, there is a non-empty set
$$\mathcal{A}=\{h\, |\, h:X\to X^\prime \text{ is not a monomorphism and\,} F(h)(x)\neq0\}.$$
$\mcc$ is of finite representation type and so Theorem \ref{main} implies that $\rad^\infty_{\mcc}=0$. Hence there exists an integer $m\geq1$ such that $\rad^m_{\mcc}=0$ and for any morphism $h\in\mathcal{A}$, there is $0<d_h<m$ such that $h\in\rad^{d_h}_{\mcc}(X,X^\prime)\backslash \rad^{d_h+1}_{\mcc}(X,X^\prime)$. Let $d=\mathrm{max}\{d_h\,|\,h\in\mathcal{A}\}$. Consider the morphism $\hat{h}:X\to X^\prime$ in $\mathcal{A}$ such that $\hat{h}\in\rad^{d}_{\mcc}(X,X^\prime)\backslash \rad^{d+1}_{\mcc}(X,X^\prime)$. Suppose that $X^\prime=\bigoplus_{i=1}^kX^\prime_i$ is the decomposition of $X^\prime$ into indecomposable direct summands. Now, we can write
 $$\hat{h}\coloneqq\left[\begin{array}{lll}\hat{h}_1\\ \vdots\\\hat{h}_k\end{array}\right]:X\longrightarrow \bigoplus_{i=1}^kX^\prime_i.$$
It is obvious that for any $i\in\{1,\dots,k\}$, $\hat{h}_i\in\rad^{d}_{\mcc}(X,X^\prime_i)$. The morphism $\hat{h}$ belongs to $\mathcal{A}$, so $F(\hat{h})(x)\neq0$. Therefore, there exists $j\in\{1,\dots,k\}$ such that $F(\hat{h}_j)(x)\neq0$. We claim that $\hat{h}_j:X\to X^\prime_j$ is a desired morphism. Indeed, we claim that $F(\hat{h}_j)(x)$ is strongly minimal in $F(X^\prime_j)$. Assume that $g:X^\prime_j\to X^{\prime\prime}$ is not a monomorphism. If $F(g)F(\hat{h}_j)(x)\neq0$, then $g\hat{h}_j\in\mathcal{A}$, since $g\hat{h}_j$ is not a monomorphism. $g$ is not a section, so by Lemma \ref{lemrad}, $g\in\rad_{\mcc}(X^\prime_j,X^{\prime\prime})$. On the other hands, $\hat{h}_j\in\rad^{d}_{\mcc}(X,X^\prime_j)$, therefore $g\hat{h}_j\in\rad^{d+1}_{\mcc}(X,X^{\prime\prime})$. But this contradicts the maximality of $d$. Hence, $F(g)F(\hat{h}_j)(x)=0$ and the claim is proved.
\end{proof}
Now, we are ready to prove the second main result in this section.
 \begin{theorem}\label{th11}
Let $\mcc$ be a functorially finite subcategory of $\modd\Lambda$ which contains $\Lambda$ or $D\Lambda$. Then $\mcc$ is of finite representation type if and only if the following
conditions hold.
\begin{itemize}
\item[(i)]
Condition \ref{c1} is satisfied.
\item[(ii)]
Any family of monomorphisms between indecomposable objects in $\mathcal{C}$ is noetherian.
\end{itemize}
\end{theorem}
\begin{proof}
The result follows from Propositions \ref{pro3}, \ref{pro5} and \ref{pro7}.
\end{proof}
\begin{remark}
By Lemma \ref{lem0} and Proposition \ref{pro2}, Condition \ref{c1} is always satisfied for $\modd\Lambda$. Therefore, Theorem \ref{th11} implies that any family of monomorphisms (resp. epimorphisms) between indecomposable $\Lambda$-modules is noetherian (resp. conoetherian) if and only if $\modd\Lambda$ is of finite representation type. This result have been already proved by Auslander (see \cite[Theorem 3.1]{Au2} and \cite[Theorem A]{AuL}).
\end{remark}
Now, as a consequence of Theorem \ref{main}, we have the following corollary.
\begin{corollary}
The following are equivalent for a functorially finite subcategory $\mcc$ of $\modd\Lambda$ which contains $\Lambda$ and $D\Lambda$.
\begin{itemize}
\item[(a)]
$\mcc$ is of finite representation type.
\item[(b)]
There exists $t\in\mathbb{N}$ such that for any indecomposable projective $\Lambda$-module $P$, $\rad^t_{\mcc}(P,-)=0$.
\item[$(\mathrm{b'})$]
There exists $t\in\mathbb{N}$ such that for any indecomposable injective $\Lambda$-module $I$,\\ $\rad^t_{\mcc}(-,I)=0$.
\item[(c)]
For any indecomposable projective $\Lambda$-module $P$, there exists $t\in\mathbb{N}$ such that $\rad^t_{\mcc}(P,-)=0$.
\item[$(\mathrm{c'})$]
For any indecomposable injective $\Lambda$-module $I$, there exists $t\in\mathbb{N}$ such that\\ $\rad^t_{\mcc}(-,I)=0$.
\end{itemize}
\end{corollary}
\begin{proof}
$(\mathrm{a})\Rightarrow(\mathrm{b})$ It follows by the part $(\mathrm{a})$ of Theorem \ref{main}.

$(\mathrm{b})\Rightarrow(\mathrm{c})$ It is clear.

$(\mathrm{c})\Rightarrow(\mathrm{a})$ Assume that $\mcc$ is not of finite representation type. By Theorem \ref{main}, $\rad^\infty_{\mcc}\neq0$ and so there exist two $\Lambda$-modules $X,Y\in\mcc$ and a nonzero morphism $f\in\rad^\infty_{\mcc}(X,Y)$. Consider a projective cover $\pi_X:P_X\longrightarrow X$ of $X$. It is obvious that $f\pi_X$ is nonzero and lies in $\rad^\infty_{\mcc}(P_X,Y)$ and there is an indecomposable direct summand $P$ of $P_X$ such that $\rad^\infty_{\mcc}(P,Y)\neq0$. So $\rad^t_{\mcc}(P,Y)\neq0$, for any $t\in\mathbb{N}$. But this contradicts the hypothesis. Therefore, $\mcc$ is of finite representation type.
Dually, we can see that the implications $(\mathrm{a}), (\mathrm{b'})$ and $(\mathrm{c'})$ are equivalent.
\end{proof}

For showing another main theorem in this section, first we prove the following proposition that gives a description of morphisms in $\rad_{\mcc}^d(A,B)$ for two indecomposable $\Lambda$-modules $A$ and $B$ in an idempotent complete subcategory $\mcc$ of $\modd\Lambda$ and a positive integer $d\geq 2$ in terms of irreducible morphisms in $\mcc$.
\begin{proposition}\label{pro-rad}
Let $\mcc$ be an idempotent complete subcategory of $\modd\Lambda$. Assume that $A,B\in\ind\mcc$ and $f\in\rad^d_{\mcc}(A,B)$ with $d\geq2$. Then we have the following statements.
\begin{itemize}
\item[(a)]
If $\mcc$ has minimal right almost split
morphisms then there exist $s\in\mathbb{N}$, $\Lambda$-modules $X_1,\dots,X_s\in\ind\mcc$, morphisms $h_i\in\rad_{\mcc}(A,X_i)$ and morphisms $g_i:X_i\longrightarrow B$, where $g_i$ is a sum of compositions of $d-1$ irreducible morphisms in $\mcc$ between indecomposable $\Lambda$-modules such that $f=\sum_{i=1}^s g_ih_i$.

Moreover, if $f\notin\rad^{d+1}_{\mcc}( A,B)$, then at least one of the $h_i$ is irreducible in $\mcc$ and $f$ can be written as $f=u+v$, where $u$ is a sum of compositions of $d$ irreducible morphisms in $\mcc$, $u\neq0$ and $v\in\rad_{\mcc}^{d+1}(A,B)$.
\item[(b)]
If $\mcc$ has minimal left almost split morphisms then there exist $s\in\mathbb{N}$, $\Lambda$-modules $X_1,\dots,X_s\in\ind\mcc$, morphisms $g_i\in\rad_{\mcc}(X_i,B)$ and morphisms $h_i:A\longrightarrow X_i$, where $h_i$ is a sum of compositions of $d-1$ irreducible morphisms in $\mcc$ between indecomposable $\Lambda$-modules such that $f=\sum_{i=1}^s g_ih_i$.

Moreover, if $f\notin\rad^{d+1}_{\mcc}( A,B)$, then at least one of the $g_i$ is irreducible in $\mcc$ and $f$ can be written as $f=u+v$, where $u$ is a sum of compositions of $d$ irreducible morphisms in $\mcc$, $u\neq0$ and $v\in\rad_{\mcc}^{d+1}(A,B)$.
\end{itemize}
\end{proposition}
\begin{proof}
By induction on $d$, we only prove $($a$)$. The proof of $($b$)$ is similar.

First, assume that $d=2$. By assumption, there exists a minimal right almost split morphism $g:X\longrightarrow B$ in $\mcc$.
Since $f$ lies in $\rad_{\mcc}(A,B)$, by Lemma \ref{lemrad}, it is not a retraction. Therefore, there exists a morphism $h:A\longrightarrow X$ such that the following diagram commutes.
\begin{center}
\scalebox{.75}{
\begin{tikzpicture}
\node (X1) at (0,0) {$A$};
\node (X2) at (4,0) {$B$.};
\node (X3) at (2,-2) {$X$};
\draw [->,thick] (X1) -- (X2) node [midway,above] {$f$};
\draw [->,thick,dashed] (X1) -- (X3) node [midway,left] {$h\,\, $};
\draw [->,thick] (X3) -- (X2) node [midway,right] {$g$};
\end{tikzpicture}}
\end{center}
Consider the decomposition $X=\bigoplus_{i=1}^sX_i$ of $X$ into indecomposable direct summands. Since $\mcc$ is closed under direct summands, each $X_i$ belongs to $\mcc$. We have $f=\sum_{i=1}^s g_ih_i$ where $g=(g_1,\dots,g_s)$ and $h=(h_1,\dots,h_s)^t$. Theorem \ref{irr-as} implies that $g_i$ is irreducible in $\mcc$, for each $1\leq i\leq s$. By assumption, $f\in\rad_{\mcc}^{2}(A,B)$ and so by Lemma \ref{irr-rad}, $f$ is not irreducible in $\mcc$. Therefore, $h_i\in\rad_{\mcc}(A,X_i)$, for each $1\leq i\leq s$. If $f\notin\rad_{\mcc}^{3}(A,B)$, then there exists at least one $i\in\{1,\dots,s\}$ such that $h_i$ does not belong to $\rad_{\mcc}^2(A,X_i)$. By Lemma \ref{irr-rad}, $h_i$ is irreducible in $\mcc$.

Now, suppose that $d\geq 3$ and the result follows for $d-1$. Let $f\in\rad^d_{\mcc}(A,B)$. There exist a $\Lambda$-module $Y$ in $\mcc$, morphisms $\alpha\in\rad_{\mcc}(A,Y)$ and $\alpha^\prime\in\rad^{d-1}_{\mcc}( Y,B)$ such that $f=\alpha^\prime \alpha$. Consider  the decomposition $Y=\bigoplus_{i=1}^r Y_i$ of $Y$ into indecomposable direct summands. Assume that $\alpha^\prime_i$ is the restriction of $\alpha^\prime$ to $Y_i$ and $\alpha_i$ is the morphism induced by the decomposition $Y=\bigoplus_{i=1}^r Y_i$. We have $f=\sum_{i=1}^r \alpha^\prime_i\alpha_i$ .
It is clear that $\alpha^\prime_i\in\rad^{d-1}_{\mcc}(Y_i,B)$, for each $i\in\{1,\dots,r\}$. Now, we apply the induction hypothesis for each $i$ and we obtain $t_i\in\mathbb{N}$, $\Lambda$-modules $Z_{i1},\dots, Z_{it_i}\in\ind\mcc$, morphisms $\beta_{ij}\in\rad_{\mcc}(Y_i,Z_{ij})$ and morphisms $\beta^\prime_{ij}:Z_{ij}\longrightarrow B$ with each $\beta^\prime_{ij}$ a sum of compositions of $d-2$ irreducible morphisms in $\mcc$ between indecomposable $\Lambda$-modules such that $\alpha^\prime_i=\sum_{j=1}^{t_i}\beta^\prime_{ij}\beta_{ij}$, for each $i\in\{1,\dots,r\}$. Since for each $i,j$, the morphism $\beta_{ij}\alpha_i:A\longrightarrow Z_{ij}$ belongs to $\rad_{\mcc}^2(A,Z_{ij})$, by the induction hypothesis, we have $\beta_{ij}\alpha_i=\sum_{k=1}^{l_{ij}}\gamma^\prime_{ijk}\gamma_{ijk}$, where $l_{ij}\geq1$ is a natural number, $\gamma^\prime_{ijk}:W_{ijk}\longrightarrow Z_{ij}$ is an irreducible morphism in $\mcc$ between indecomposable $\Lambda$-modules and $\gamma_{ijk}\in\rad_{\mcc}(A,W_{ijk})$, for each $k$. Therefore, we have
$$f=\alpha^\prime \alpha=\sum_{i=1}^r \alpha^\prime_i\alpha_i=\sum_{i=1}^r(\sum_{j=1}^{t_i}\beta^\prime_{ij}\beta_{ij})\alpha_i=\sum_{i=1}^r\sum_{j=1}^{t_i}\beta^\prime_{ij}(\beta_{ij}\alpha)=\sum_{i=1}^r\sum_{j=1}^{t_i}\sum_{k=1}^{l_{ij}}\beta^\prime_{ij}\gamma^\prime_{ijk}\gamma_{ijk},$$
that $W_{ijk}\in\ind\mcc$, $\gamma_{ijk}\in\rad_{\mcc}(A,W_{ijk})$ and each $\beta^\prime_{ij}\gamma^\prime_{ijk}:W_{ijk}\longrightarrow B$ is a sum of compositions of $d-1$ irreducible morphisms in $\mcc$, for any $i, j, k$. Now, assume that $f$ does not belong to $\rad_{\mcc}^{d+1}(A,B)$. Therefore, at least one of the $\gamma_{ijk}$ does not belong to $\rad_{\mcc}^2(A,W_{ijk})$ and so by Lemma \ref{irr-rad}, it is irreducible and the result follows.
\end{proof}

Now, we are ready to prove the final main theorem of this section.

\begin{theorem}\label{sumirr}
Let $\mcc$ be a functorially finite subcategory of $\modd\Lambda$ which contains $\Lambda$ or $D\Lambda$ and $f\in\rad_{\mcc}( A,B)$ be a nonzero morphism with $A,B\in\ind\mcc$. If $\mcc$ is of finite representation type, then $f$ is a sum of compositions of irreducible morphisms in $\mcc$ between indecomposable $\Lambda$-modules.
\end{theorem}
\begin{proof}
By Remark \ref{mlra}, there exist minimal right almost split morphisms in $\mcc$ and so the conditions of Proposition \ref{pro-rad} are satisfied. We know that there exists $l\in\mathbb{N}$ such that $\rad_{\mcc}^\infty(A,B)=\rad_{\mcc}^l(A,B)$ (see the part $(\mathrm{v})\Rightarrow(\mathrm{ii})$ of the proof of Theorem \ref{main}).
Since $\mcc$ is of finite representation type, we conclude from Theorem \ref{main} that $\rad_{\mcc}^\infty=0$ and so $\rad_{\mcc}^l(A,B)=0$. Therefore, there exists $t< l$ such that $f\in\rad_{\mcc}^t(A,B)\backslash\rad_{\mcc}^{t+1}(A,B)$. If $t=1$, the result follows by Lemma \ref{irr-rad}. Assume that $t\geq2$. By Proposition \ref{pro-rad}, we have $f=u_0+v_1$, where $u_0$ is a sum of compositions of irreducible morphisms in $\mcc$ between indecomposable $\Lambda$-modules and $v_1\in\rad_{\mcc}^{t+1}(A,B)$. If $v_1$ is nonzero, then we repeat the same procedure with $v_1$ and obtain $v_1=u_1+v_2$, where $u_1$ is the sum of compositions of irreducible morphisms in $\mcc$ between indecomposable $\Lambda$-modules and $v_2\in\rad_{\mcc}^{s+1}(A,B)$ with $s>t$. Therefore, we can write $f=u_0+u_1+v_2$. If $v_2$ is nonzero, we repeat this procedure again. This procedure stops after finitely many times because $\rad_{\mcc}^l(A,B)=0$. Finally, we obtain that $f$ is a sum of compositions of irreducible morphisms in $\mcc$ between indecomposable $\Lambda$-modules and the result follows.
\end{proof}

\section{The depth of morphisms and the representation type of functorially finite subcategories}
In this section, we provide other two criteria for recognizing the finiteness of the representation type of a functorially finite subcategory $\mcc$ of $\modd\Lambda$ which contains $\Lambda$ or $D\Lambda$ in terms of the depth of some special morphisms in $\mcc$. Also, by using the depth of the finite number of special morphisms in $\mcc$, we give a nilpotency index of $\rad_{\mcc}$, which is independent from the maximal length of indecomposable $\Lambda$-modules in $\mcc$.

The depth of a morphism in $\modd\Lambda$ was defined by Chaio and Liu in \cite{Ch-Li}.
We can extend this concept to each morphism in an additive category.

\begin{definition}\label{def4.1}
Let $\mathcal{C}$ be an additive category and $f:X\rightarrow Y$ be a morphism in $\mathcal{C}$. We say that the depth of $f$ in $\mathcal{C}$ is infinite if $f\in\rad^\infty_\mathcal{C}(X,Y)$. Otherwise, there exists a unique integer $t\geq 0$ such that $f\in\rad^t_\mathcal{C}(X,Y)\backslash\rad^{t+1}_\mathcal{C}(X,Y)$. In this case, we say that the depth of $f$ in $\mathcal{C}$ is $t$. We denote the depth of $f$ in $\mathcal{C}$ by $\dep_\mathcal{C}(f)$.
\end{definition}
Let $\mcc$ be a functorially finite subcategory of $\modd\Lambda$. Throughout this section, we fix the right $\mcc$-approximation $r_S:M_S^r\rightarrow S$, the left $\mcc$-approximation $l_S:S\rightarrow M_S^l$, the projective cover $\pi_S:P_S\rightarrow S$ and the injective envelope $\iota_S:S\rightarrow I_S$ for each simple $\Lambda$-module $S$.
Also, we denote the composition $\iota_S\pi_S$ by $\theta_S$, for every simple $\Lambda$-module $S$.

We recall that the (Jacobson) radical $\rad M$ of the $\Lambda$-module $M$ is the intersection of all the maximal submodules of $M$ and for the nonzero $\Lambda$-module $M$, $\mathrm{top}M\coloneqq M/\rad M$ is a semisimple $\Lambda$-module.

Before the proof of the first main theorem in this section, we prove the following easy lemma.
\begin{lemma}\label{lem}
Let $\mcc$ be a subcategory of $\modd\Lambda$ and $f\in\rad^\infty_{\mcc}(X,Y)$ for two $\Lambda$-modules $X,Y\in\mcc$. Then, for any $\Lambda$-module $Z\in\mcc$ and any nonzero morphism $g:\im f{\longrightarrow} Z$, the composition of morphisms $X\overset{f}{\longrightarrow} \im f \overset{g}{\longrightarrow}Z$
belongs to $\rad^\infty_{\mcc}(X,Z)$.
 \end{lemma}
 \begin{proof}
We show that $gf\in\rad^i_{\mcc}(X,Z)$, for any $i\in\mathbb{N}$. Fix $t\in\mathbb{N}$ and put $X_0\coloneqq X$ and $X_{t+1}\coloneqq Y$. Since $f:X\to Y$ belongs to $\rad^\infty_{\mcc}(X,Y)$, $f$ belongs to $\rad^{t+1}_{\mcc}(X,Y)$. Therefore, there exist $\Lambda$-modules $X_1,\dots,X_t\in\mcc$ and morphisms $h_i\in\rad_{\mcc}(X_{i-1},X_i)$, for $i=1,\dots,t+1$ such that $f=h_{t+1}\dots h_1$. Indeed, we have the following factorization.
\begin{center}
\scalebox{.8}{
\begin{tikzpicture}
  \node (X1) {$X_1$};
  \node (X2) [node distance=2cm, right of=X1] {$X_2$};
  \node (dot) [node distance=2cm, right of=X2] {$\cdots$};
  \node (xt1) [node distance=2cm, right of=dot] {$X_{t-1}$};
    \node (xt) [node distance=2cm, right of=xt1] {$X_t$};
  \node (X) [node distance=1.4cm, left of=X1, above of=X1] {$X_0=X$};
   \node (Y) [node distance=1.4cm, right of=xt, above of=xt] {$X_{t+1}=Y$};
     \draw[->] (X1) to node [below] {$h_2$} (X2);
       \draw[->] (X2) to node {} (dot);
         \draw[->] (dot) to node {} (xt1);
           \draw[->] (xt1) to node [below] {$h_t$} (xt);
             \draw[->] (X) to node [left] {$h_1$} (X1);
               \draw[->] (xt) to node [right] {$\,h_{t+1}$} (Y);
  \draw[->] (X) to node [above] {$f$} (Y);
\end{tikzpicture}}
\end{center}

It is clear that $\im f=\im (h_{t+1}\dots h_1)$ and we can consider the composition of morphisms
 \begin{equation}
X= X_0 \overset{h_1}{\longrightarrow} X_1 \overset{h_2}{\longrightarrow}X_2\longrightarrow\cdots \longrightarrow X_{t-1} \overset{h_t}{\longrightarrow} X_t \overset{h_{t+1}}{\longrightarrow}\im (h_{t+1}\dots h_1)\overset{g}{\longrightarrow} Z. \notag
\end{equation}
It is obvious that $h_t\dots h_1\in\rad^t_{\mcc}(X_1,X_t)$ and $gh_{t+1}\in\Hom_{\mcc}(X_t,Z)$. Hence, $gf=gh_{t+1}\dots h_1\in\rad^t_{\mcc}(X,Z)$.
 \end{proof}
\begin{theorem}\label{dp1}
Let $\mcc$ be a functorially finite subcategory of $\modd\Lambda$.
\begin{itemize}
\item[(a)]
 If $\mcc$ contains $\Lambda$, then the following are equivalent.
 \begin{itemize}
\item[(i)]
$\mcc$ is of finite representation type.
\item[(ii)]
The depth of the composition $l_S\pi_S$ in $\mcc$ is finite for every simple $\Lambda$-module $S$.
\item[(iii)]
For any indecomposable projective $\Lambda$-module $P$, $\rad^\infty_{\mcc}(P,-)=0$.
\end{itemize}
\item[(b)]
If $\mcc$ contains $D\Lambda$, then the following are equivalent.
\begin{itemize}
\item[(i)]
$\mcc$ is of finite representation type.
\item[(ii)]
The depth of the composition $\iota_Sr_S$ in $\mcc$ is finite for every simple $\Lambda$-module $S$.
\item[(iii)]
For any indecomposable injective $\Lambda$-module $I$, $\rad^\infty_{\mcc}(-,I)=0$.
\end{itemize}
\end{itemize}
\end{theorem}
\begin{proof}
We only prove $(\mathrm{a})$, because the proof of $(\mathrm{b})$ is dual.

$(\mathrm{i})\Rightarrow(\mathrm{ii})$
Assume that $\mcc$ is of finite representation type. By Theorem \ref{main}(b), $\rad^\infty_{\mcc}=0$. We know that for every simple $\Lambda$-module $S$, $\pi_S$ is an epimorphism and $l_S$ is nonzero. So $l_S\pi_S$ is nonzero and does not lie in $\rad^\infty_{\mcc}$. Therefore, by the definition of the depth of a morphism in $\mcc$, $\dep_{\mcc}(l_S\pi_S)$ is finite.

$(\mathrm{ii})\Rightarrow(\mathrm{iii})$ Assume that there is an indecomposable projective $\Lambda$-module $P$ such that $\rad^\infty_{\mcc}(P,-)\neq0$. Therefore, there exist $M\in\mcc$ and a nonzero morphism $f\in\rad^\infty_{\mcc}(P,M)$. Suppose that $S$ is a simple summand of $\mathrm{top}(\im f)$. Then there is a nonzero epimorphism $\pi:\im f\to S$. Consider the following diagram.

\begin{center}
\scalebox{.8}{
\begin{tikzpicture}
\node (X1) at (0,3) {$P_S$};
\node (X2) at (0,1) {$S$};
\node (X3) at (-3,1) {$\im f$};
\node (X4) at (-6,1) {$P$};
\node (X5) at (0,-1) {$M^l_S$};
\draw [->,thick] (X1) -- (X2) node [midway,right] {$\pi_S$};
\draw [->,thick] (X3) -- (X2) node [midway,below] {$\pi$};
\draw [->,thick] (X4) -- (X3) node [midway,below] {$f$};
\draw [->,thick] (X2) -- (X5) node [midway,right] {$l_S$};
\end{tikzpicture}}
\end{center}
Since $\pi f$ is an epimorphism and  $P_S$ is a projective module, we have the following commutative diagram:
\begin{center}
\scalebox{.8}{
\begin{tikzpicture}
\node (X1) at (0,3) {$P_S$};
\node (X2) at (0,1) {$S$};
\node (X3) at (-3,1) {$\im f$};
\node (X4) at (-6,1) {$P$};
\node (X5) at (0,-1) {$M^l_S$};
\draw [->,thick] (X1) -- (X2) node [midway,right] {$\pi_S$};
\draw [->,thick] (X3) -- (X2) node [midway,below] {$\pi$};
\draw [->,thick] (X4) -- (X3) node [midway,below] {$f$};
\draw [->,thick] (X2) -- (X5) node [midway,right] {$l_S$};
\draw [->,thick,dashed] (X1) -- (X4) node [midway,above] {$ h$};
\end{tikzpicture}}
\end{center}
Indeed, there exists a morphism $h$ such that $\pi_S=\pi fh$ and so we obtain $l_S\pi_S=l_S\pi fh$. By Lemma \ref{lem}, $l_S\pi f\in\rad^\infty_{\mcc}(P,M_S^l)$ and so $l_S\pi_S=l_S\pi fh\in\rad^\infty_{\mcc}(P_S,M_S^l)$. This means that $\dep_{\mcc}(l_S\pi_S)$ is infinite  and this is a contradiction.

$(\mathrm{iii})\Rightarrow(\mathrm{i})$ Assume that $\mcc$ is not of finite representation type. By Theorem \ref{main}(b), $\rad^\infty_{\mcc}\neq0$ and so there exist two $\Lambda$-module $X,Y\in\mcc$ and a nonzero morphism $f\in\rad^\infty_{\mcc}(X,Y)$. Consider a projective cover $\pi_X:P_X\rightarrow X$ of $X$. It is obvious that $f\pi_X$ is nonzero and lies in $\rad^\infty_{\mcc}(P_X,Y)$ and there is an indecomposable direct summand $P$ of $P_X$ such that  $\rad^\infty_{\mcc}(P,Y)\neq0$. But this contradicts the hypothesis. Therefore, $\mcc$ is of finite representation type.
\end{proof}
\begin{lemma}\label{lem01}
Let $\mcc$ be a functorially finite subcategory of $\modd\Lambda$ which contains $\Lambda$ and $D\Lambda$. Then $\mcc$ is of finite representation type if and only if there exists an integer $d\geq1$ such that $\rad_{\mcc}^d(P_S,I_S)=0$, for any simple $\Lambda$-module $S$.
\end{lemma}
\begin{proof}
Suppose that $\mcc$ is of finite representation type. Then by the proof of Theorem \ref{main}, for the Harada-Sai bound $2^b-1$ we have $\rad_{\mcc}^{2^b-1}=0$ and the proof is complete.

Conversely, assume that there exists an integer $d\geq1$ such that $\rad_{\mcc}^d(P_S,I_S)=0$, for any simple $\Lambda$-module $S$ and $\mcc$ is not of finite representation type. By Theorem \ref{main}, $\rad_{\mcc}^\infty\neq0$. Thus, there exist two $\Lambda$-modules $M$ and $N$ in $\mcc$ such that $\rad_{\mcc}^\infty(M,N)$ is nonzero and consequently $\rad_{\mcc}^d(M,N)$ is nonzero. Now, consider the nonzero morphism $f:M\rightarrow N$ in $\rad_{\mcc}^d(M,N)$ and let $S$ be a simple composition factor of $\im(f)$. By \cite[Proposition 2.4]{Ch-Li}, there exist morphisms $g:P_S\rightarrow M$ and $h:N\rightarrow I_S$ such that $hfg=\theta_S$ and hence the nonzero morphism $\theta_S$ belongs to $\rad_{\mcc}^d(P_S,I_S)$, which is a contradiction.
\end{proof}

Now, by using the depth of $\theta_S$ for every simple $\Lambda$-module $S$, we give another criterion for recognizing the finiteness of the representation type of the functorially finite subcategory $\mcc$ of $\modd\Lambda$ and find a good nilpotency index of the radical of $\mcc$.

\begin{theorem}\label{dp2}
Let $\mcc$ be a functorially finite subcategory of $\modd\Lambda$ which contains $\Lambda$ and $D\Lambda$. Then $\mcc$ is of finite representation type if and only if the depth of $\theta_S$ in $\mcc$ is finite, for every simple $\Lambda$-module $S$. Moreover, in this case, if $m$ is the maximal depth of the morphisms $\theta_S$ in $\mcc$ with $S$ ranging over all simple $\Lambda$-modules, then $m+1$ is a nilpotency index of $\rad_{\mcc}$.
\end{theorem}
\begin{proof}
Assume that $\mcc$ is of finite representation type. By Theorem \ref{main}, $\rad_{\mcc}^\infty=0$.
We know that $\theta_S$ is nonzero, for every simple $\Lambda$-module $S$.
Therefore, $\theta_S$ does not belong to $\rad_{\mcc}^\infty$ and so $\dep_{\mcc}(\theta_S)$ is finite, for any simple $\Lambda$-module $S$.

Conversely, assume that $\dep_{\mcc}(\theta_S)$ is finite for any simple $\Lambda$-module $S$. Set
$$m\coloneqq\max\{\dep_{\mcc}(\theta_S)\,|\, S \text{\, is a simple\,\,}\Lambda\text{-module}\}.$$
We claim that $\rad_{\mcc}^{m+1}(P_S,I_S)=0$, for any simple $\Lambda$-module $S$. Assume on a contrary that there is a simple $\Lambda$-module $S^\prime$ such that $\rad_{\mcm}^{m+1}(P_{S^\prime},I_{S^\prime})$ is nonzero.
Consider a nonzero morphism $f:P_{S^\prime}\rightarrow I_{S^\prime}$ in $\rad_{\mcc}^{m+1}(P_{S^\prime},I_{S^\prime})$. By \cite[Proposition 2.2]{Ch-Li}, there exist morphisms $u\in\End_{\Lambda}(P_{S^\prime})$ and $v\in\End_{\Lambda}(I_{S^\prime})$ such that $vfu=\theta_{S^\prime}$ and so the morphism $\theta_{S^\prime}$ lies in $\rad_{\mcc}^{m+1}(P_{S^\prime},I_{S^\prime})$. But this is a contradiction since $\dep_{\mcc}(\theta_{S^\prime})\leq m$. Hence, our claim is proved and the result follows from Lemma \ref{lem01}.

Now, we show that in this case $m+1$ is a nilpotency index of $\rad_{\mcc}$.
In the above proof, we show that $\rad_{\mcc}^{m+1}(P_S,I_S)$ vanishes, for every simple $\Lambda$-module $S$.
Now, we claim that $\rad_{\mcc}^{m+1}(M,N)=0$, for $\Lambda$-modules $M$ and $N$ in $\mcc$.
Assume on a contrary that $\rad_{\mcc}^{m+1}(M,N)\neq0$ for $\Lambda$-modules $M$ and $N$ in $\mcc$ and consider a nonzero morphism $f:M\rightarrow N$ in $\rad_{\mcc}^{m+1}(M,N)$. Suppose that $S$ is a simple composition factor of $\im(f)$. By \cite[Proposition 2.4]{Ch-Li}, there exist morphisms $g:P_S\rightarrow M$ and $h:N\rightarrow I_S$ such that $hfg=\theta_S$ and so the morphism $\theta_S$ lies in $\rad_{\mcc}^{m+1}(P_S,I_S)$. But this is a contradiction since $\theta_S$ is nonzero. Thus $\rad_{\mcc}^{m+1}(M,N)=0$, for $\Lambda$-modules $M$ and $N$ in $\mcc$ and the result follows.
\end{proof}

\section{Examples}
In this section we give some examples of functorially finite subcategories which satisfy the conditions of the main results of this paper.
\subsection{$n$-cluster tilting subcategories}
Throughout this subsection, $n$ is a fixed positive integer. The $n$-cluster tilting subcategories are introduced and investigated by Iyama in \cite{I2, I3, I1}.
\begin{definition}$($\cite[Definition 2.2]{I1}$)$
A functorially finite full subcategory $\mathcal{M}$ of $\modd\Lambda$ is called \textit{$n$-cluster tilting} if
\begin{align*}
\mathcal{M}&=\{X\in\text{$\modd\Lambda$}\,|\,\mathrm{Ext}^i_\Lambda(X,\mathcal{M})=0, \,\,\, \text{for } 0<i<n\}\\
&=\{X\in\text{$\modd\Lambda$}\,|\,\mathrm{Ext}^i_\Lambda(\mathcal{M},X)=0, \,\,\, \text{for } 0<i<n\}.
\end{align*}
\end{definition}

Note that $\modd\Lambda$ itself is the unique $1$-cluster tilting subcategory of $\modd\Lambda$. It is obvious that $\mcm$ is closed under direct sums and summands and isomorphisms and it contains $\Lambda$ and $D\Lambda$.

\begin{definition}$($\cite[Definition 2.5]{HJV} and \cite[Definition 2.13]{EN}$)$\label{def-rad}
If $\mathcal{M}$ is an $n$-cluster tilting subcategory of $\modd\Lambda$, then the pair $(\Lambda,\mathcal{M})$ is called an {\em $n$-homological pair}. An $n$-homological pair $(\Lambda,\mathcal{M})$ is called of {\em finite type} if $\mathcal{M}$ has an additive generator or equivalently the number of isomorphism classes of indecomposable objects in $\mathcal{M}$ is finite.
\end{definition}

As a consequence of Theorem \ref{main}, we have the following corollary.
\begin{corollary}
Let $(\Lambda,\mcm)$ be an $n$-homological pair. Then the following are equivalent.
\begin{itemize}
\item[$(a)$]
There exists $t\in\mathbb{N}$ such that for any $X\in\ind\mcm$, $\rad^t_{\mcm}(X,-)=0$.
\item[$(a')$]
There exists $t\in\mathbb{N}$ such that for any $X\in\ind\mcm$, $\rad^t_{\mcm}(-,X)=0$.
\item[$(b)$]
For any $X\in\ind\mcm$, there exists $t\in\mathbb{N}$ such that $\rad^t_{\mcm}(X,-)=0$.
\item[$(b')$]
For any $X\in\ind\mcm$, there exists $t\in\mathbb{N}$ such that $\rad^t_{\mcm}(-,X)=0$.
\item[$(c)$]
Any family of morphisms between indecomposable $\Lambda$-modules in $\mcm$ is noetherian.
\item[$(c')$]
Any family of morphisms between indecomposable $\Lambda$-modules in $\mcm$ is conoetherian.
\item[$(d)$]
$(\Lambda,\mcm)$ is of finite type.
\item[$(e)$]
$\rad^\infty_{\mcm}=0$.
\end{itemize}
\end{corollary}

For $n=1$, according to the Auslander's works, we know that $\modd\Lambda$ is of finite representation type if and only if $\rad^\infty_{\Lambda}=0$ if and only if every family of monomorphisms (resp. epimorphisms) between indecomposable objects in $\modd\Lambda$ is noetherian (resp. conoetherian). Motivated by these results we pose the following question.

\begin{question}\label{q1}
Let $(\Lambda,\mcm)$ be an $n$-homological pair with $n\geq2$. Assume that every family of monomorphisms (resp. epimorphisms) between indecomposable $\Lambda$-modules in $\mcm$ is noetherian (resp. conoetherian). Is the $n$-homological pair $(\Lambda,\mcm)$ of finite type?
\end{question}
As a consequence of Theorem \ref{th11}, we have the following result.
 \begin{corollary}\label{cor3}
 Let $(\Lambda,\mcm)$ be an $n$-homological pair with $n\geq2$. Then $(\Lambda,\mcm)$ is of finite 
type if and only if 
\begin{itemize}
\item[(i)]
Condition \ref{c1} is satisfied for $\mcm$.
\item[(ii)]
Any family of monomorphisms between indecomposable objects in $\mcm$ is noetherian.
\end{itemize}
\end{corollary}
Therefore, we can reduce Question \ref{q1} to the following question.
\begin{question}
Let $(\Lambda,\mcm)$ be an $n$-homological pair with $n\geq2$.
Is Condition \ref{c1} satisfied for $\mcm$?
\end{question}
As consequences of Theorems \ref{sumirr} and \ref{dp1}, we have the following results.
\begin{corollary}
Let $(\Lambda,\mcm)$ be an $n$-homological pair of finite type and $f\in\rad_{\mcm}( A,B)$ be a nonzero morphism with $A,B\in\ind\mcm$. Then $f$ is a sum of compositions of irreducible morphisms in $\mcm$ between indecomposable $\Lambda$-modules.
\end{corollary}

\begin{corollary}\label{dp11}
The following statements are equivalent for an $n$-homological pair $(\Lambda,\mcm)$.
\begin{itemize}
\item[(a)]
$(\Lambda,\mcm)$ is of finite type.
\item[(b)]
The depth of the composition $l_S\pi_S$ in $\mcm$ is finite for every simple $\Lambda$-module $S$.
\item[(c)]
For any indecomposable projective $\Lambda$-module $P$, $\rad^\infty_{\mcm}(P,-)=0$.
\item[(d)]
The depth of the composition $\iota_Sr_S$ in $\mcm$ is finite for every simple $\Lambda$-module $S$.
\item[(e)]
For any indecomposable injective $\Lambda$-module $I$, $\rad^\infty_{\mcm}(-,I)=0$.
\end{itemize}
\end{corollary}

Corollary \ref{dp11} is the higher dimensional analog of \cite[Theorem 2.7]{Ch-Li}. In fact, in the classical case, we know that $\modd\Lambda$ is the unique $1$-cluster tilting subcategory of $\modd\Lambda$ and we can consider identity morphisms as left and right approximations. Note that some other similar results were proved for the classical case in \cite{Ch,CMT,L}.

Now, as a consequence of Theorem \ref{dp2}, we have the following corollary.

\begin{corollary}\label{dp21}
An $n$-homological pair $(\Lambda,\mcm)$ is of finite type if and only if the depth of $\theta_S$ in $\mcm$ is finite, for every simple $\Lambda$-module $S$. Moreover, in this case, if $m$ is the maximal depth of the morphisms $\theta_S$ in $\mcm$ with $S$ ranging over all simple $\Lambda$-modules, then $m+1$ is a nilpotency index of $\rad_{\mcm}$.
\end{corollary}

The following example shows that the nilpotency index in Corollary \ref{dp21} is better than the Harada--Sai bound.

\begin{example}
Let $\Lambda$ be the algebra given by the quiver
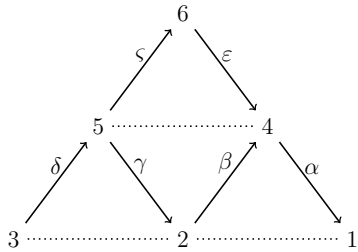
\begin{center}
\scalebox{.75}{
\begin{tikzpicture}
\node (X3) at (0,0) {$3$};
\node (X2) at (3,0) {$2$};
\node (X1) at (6,0) {$1$};
\node (X5) at (1.5,2) {$5$};
\node (X4) at (4.5,2) {$4$};
\node (X6) at (3,4) {$6$};
\draw [->,thick] (X3) -- (X5)node [midway,above] {$\delta$};
\draw [->,thick] (X5) -- (X2)node [midway,above] {$\gamma$};
\draw [->,thick] (X2) -- (X4)node [midway,above] {$\beta$};
\draw [->,thick] (X4) -- (X1)node [midway,above] {$\,\alpha$};
\draw [->,thick] (X5) -- (X6)node [midway,above] {$\varsigma$};
\draw [->,thick] (X6) -- (X4)node [midway,above] {$\,\varepsilon$};
\draw [-,thick,dotted] (X5) -- (X4);
\draw [-,thick,dotted] (X3) -- (X2);
\draw [-,thick,dotted] (X2) -- (X1);
\end{tikzpicture}}
\end{center}
bound by $\delta\gamma=\beta\alpha=\gamma\beta+\zeta\varepsilon=0$. The algebra $\Lambda$ is $2$-representation finite with the unique $2$-cluster tilting module $M\coloneqq \Lambda\oplus D\Lambda\oplus S_2$ (See \cite[Section B]{J} and \cite[Theorem 1.18]{I4}). Let $\mathcal{M}=\add(M)$. Then $(\Lambda,\mathcal{M})$ is an $2$-homological pair. The Auslander--Reiten quiver of $\Lambda$ is as follows and $M$ is the direct sum of all bold modules.
\begin{center}
\scalebox{.7}{
\begin{tikzpicture}
\node (X3)  at (0,0) {$\textbf{3}$};
\node (X2) at (3,0) {$5$};
\node (X1) at (6,0)  {$\begin{array}{ll}\,5\\26\end{array}$};
\node (X5) at (1.5,2) {$\begin{array}{ll}\textbf{3}\\\textbf{5}\end{array}$};
\node (X4) at (4.5,2) {$\begin{array}{ll}5\\6\end{array}$};
\node (X6) at (3,4) {$\begin{array}{lll}\textbf{3}\\\textbf{5}\\\textbf{6}\end{array}$};
\node (X7) at (4.5,-2) {$\begin{array}{ll}\textbf{5}\\\textbf{2}\end{array}$};
\node (X8) at (7.5,-2) {$6$};
\node (X12) at (10.5,-2) {$\begin{array}{ll}\textbf{2}\\\textbf{4}\end{array}$};
\node (X9) at (7.5,0) {$\begin{array}{lll}\,\textbf{6}\\\textbf{54}\\ \,\textbf{2}\end{array}$};
\node (X11)at (9,0) {$\begin{array}{ll}26\\ \,4\end{array}$};
\node (X14) at (12,0) {$4$};
\node (X17) at (15,0) {$\textbf{1}$};
\node (X10) at (7.5,2) {$\textbf{2}$};
\node (X13) at (10.5,2) {$\begin{array}{ll}6\\4\end{array}$};
\node (X16) at (13.5,2) {$\begin{array}{ll}\textbf{4}\\\textbf{1}\end{array}$};
\node (X15) at (12,4) {$\begin{array}{lll}\textbf{6}\\\textbf{4}\\\textbf{1}\end{array}$};
\draw [->,thick] (X3) -- (X5);
\draw [->,thick] (X5) -- (X2);
\draw [->,thick] (X5) -- (X2);
\draw [->,thick] (X2) -- (X4);
\draw [->,thick] (X4) -- (X1);
\draw [->,thick] (X5) -- (X6);
\draw [->,thick] (X6) -- (X4);
\draw [->,thick] (X2) -- (X7);
\draw [->,thick] (X7) -- (X1);
\draw [->,thick] (X1) -- (X8);
\draw [->,thick] (X1) -- (X9);
\draw [->,thick] (X1) -- (X10);
\draw [->,thick] (X8) -- (X11);
\draw [->,thick] (X9) -- (X11);
\draw [->,thick] (X10) -- (X11);
\draw [->,thick] (X11) -- (X13);
\draw [->,thick] (X11) -- (X12);
\draw [->,thick] (X13) -- (X14);
\draw [->,thick] (X12) -- (X14);
\draw [->,thick] (X14) -- (X16);
\draw [->,thick] (X13) -- (X15);
\draw [->,thick] (X15) -- (X16);
\draw [->,thick] (X16) -- (X17);
\end{tikzpicture}}
\end{center}
The maximal length of the indecomposable modules in $\mathcal{M}$ is equal to $4$. Therefore, the Harada--Sai bound is equal to $2^4-1 = 15$ and so $\rad^{15}_{\mcm}=0$. Consider the morphisms $\theta_{S_i}$.
\begin{flushleft}
$\theta_{S_1}:\begin{array}{lll}\textbf{6}\\\textbf{4}\\\textbf{1}\end{array}\to\begin{array}{ll}\textbf{4}\\\textbf{1}\end{array}\to\textbf{1},$\\
$\theta_{S_2}:\begin{array}{ll}\textbf{5}\\\textbf{2}\end{array}\to\begin{array}{ll}\,5\\26\end{array}\to \textbf{2}\to\begin{array}{ll}26\\ \,4\end{array}\to \begin{array}{ll}\textbf{2}\\\textbf{4}\end{array},$\\
$\theta_{S_3}:\textbf{3}\to\begin{array}{ll}\textbf{3}\\\textbf{5}\end{array}\to\begin{array}{lll}\textbf{3}\\\textbf{5}\\\textbf{6}\end{array},$\\
$\theta_{S_4}:\begin{array}{lll}\,\textbf{6}\\\textbf{54}\\ \,\textbf{2}\end{array}\to\begin{array}{ll}26\\ \,4\end{array}\to\begin{array}{ll}\textbf{2}\\\textbf{4}\end{array}\to4\to \begin{array}{ll}\textbf{4}\\\textbf{1}\end{array},$\\
$\theta_{S_5}:\begin{array}{ll}\textbf{3}\\\textbf{5}\end{array}\to5\to\begin{array}{ll}\textbf{5}\\\textbf{2}\end{array}\to\begin{array}{ll}\,5\\26\end{array}\to\begin{array}{lll}\,\textbf{6}\\\textbf{54}\\ \,\textbf{2}\end{array},$\\
$\theta_{S_6}:\begin{array}{lll}\textbf{3}\\\textbf{5}\\\textbf{6}\end{array}\to \begin{array}{ll}5\\6\end{array}\to \begin{array}{ll}\,5\\26\end{array}\to 6\to\begin{array}{ll}26\\ \,4\end{array}\to\begin{array}{ll}6\\4\end{array}\to\begin{array}{lll}\textbf{6}\\\textbf{4}\\\textbf{1}\end{array}.$
\end{flushleft}
It is obvious that for $i\in\{1,\dots,5\}$, $\dep_{\mcm}(\theta_{S_i})=2$ and $\dep_{\mcm}(\theta_{S_6})=1$. Therefore, the amount of $m$ in Theorem \ref{dp2} is equal to 2, and so $\rad^3_{\mcm}=0$.
\end{example}
\subsection{The contravariantly finite resolving subcategories}
We recall that the subcategory $\mcc$ of $\modd\Lambda$ is called {\em resolving} if $\Lambda\in\mcc$ and it is closed under direct summands, extensions and kernels of epimorphisms \cite{AB}. The contravariantly finite resolving subcategory $\mcc$ is covariantly finite (see \cite[Corollary 2.6]{KrS}), so it is functorially finite. Therefore, if $\mcc$ is a contravariantly finite resolving subcategory of $\modd\Lambda$, then Theorems \ref{main}(b), \ref{th11}, \ref{sumirr} and \ref{dp1}(a) are satisfied for it.
\subsection{The subcategory of finitely generated Gorenstein projective modules}
Let $\Lambda$ be an artin algebra. A \textit{complete projective resolution} is an exact sequence
 \begin{equation}
   (\mathcal{P}^\bullet, d) = \cdots\overset{}{\longrightarrow}P^{-1} \overset{d^{-1}}{\longrightarrow}P^{0} \overset{d^0}{\longrightarrow} P^1\overset{d^1}{\longrightarrow} P^2\overset{}{\longrightarrow}\cdots \notag
    \end{equation}
of $P^i\in\mathrm{Proj}(\Lambda)$ such that for any $P\in\mathrm{Proj}(\Lambda)$, $\Hom_{\Lambda}(-,P)$ is again exact. $M\in\Mod\Lambda$ is called \textit{Gorenstein projective} if there is a complete projective resolution $  (\mathcal{P}^\bullet, d)$ such that $M=\im d^{-1}$. We denote by $\mathrm{GProj}(\Lambda)$ the full subcategory of $\Mod\Lambda$ consisting of Gorenstein
projective modules. Similarly, a finitely generated Gorenstein projective module is defined. We denote by $\mathrm{Gproj}(\Lambda)$ the full subcategory of $\modd\Lambda$ consisting of Gorenstein projective modules. A complete injective resolution, a Gorenstein injective module, a finitely generated Gorenstein injective module and the categories $\mathrm{GInj}(\Lambda)$ and $\mathrm{Ginj}(\Lambda)$ are defined dually \cite{EJ}. The full subcategory $\mathrm{Gproj}(\Lambda)$ of $\modd\Lambda$ is an additive resolving subcategory which is closed under direct summands (See \cite[Theorem 2.4]{H}). Consider $\mathrm{Ext}$-orthogonal subcategories
\begin{align*}
(\mathrm{GProj}(\Lambda))^\perp &= \{X \in\Mod\Lambda | \Ext^n_\Lambda(M,X) = 0,\, \forall n\geq 1,\, \forall M\in \mathrm{GProj}(\Lambda)\},\\
^\perp(\mathrm{GInj}(\Lambda)) &= \{X \in\Mod\Lambda | \Ext^n_\Lambda(X,M) = 0,\, \forall n\geq 1,\, \forall M\in \mathrm{GInj}(\Lambda)\}.
\end{align*}
$\Lambda$ is called virtually Gorenstein, if
$(\mathrm{GProj}(\Lambda))^\perp =\, ^\perp(\mathrm{GInj}(\Lambda))$.
In this case, $\mathrm{Gproj}(\Lambda)$ is a functorially finite subcategory of $\modd\Lambda$ \cite{Be}.
Therefore, if $\Lambda$ is virtually Gorenstein, then Theorems \ref{main}(b), \ref{th11}, \ref{sumirr} and \ref{dp1}(a) are satisfied for $\mathrm{Gproj}(\Lambda)$.
\section*{acknowledgements}
The research of the second author was in part supported by a grant from IPM (No. 1401170217). The work of the second author is based upon research funded by Iran National Science Foundation (INSF) under project No. 4001480.

\end{document}